\documentclass{article}
\usepackage[utf8]{inputenc}
\usepackage[all,cmtip]{xy}
\usepackage{amssymb,amsmath,amsthm}
\usepackage{mathtools, mathdots}
\usepackage{fancyhdr} 
\usepackage{mathrsfs, calligra, bbm}
\usepackage[top=1.5in, bottom=1.5in, left=1in, right=1in]{geometry}
\usepackage{hyperref}
\usepackage{tikz-cd}
\usetikzlibrary{decorations.pathmorphing}
\usepackage{dirtytalk}
\usepackage{tikz}
\usepackage{enumitem}
\usepackage[T1]{fontenc}
\usepackage{graphicx}
\usepackage{subcaption}
\usepackage{array}
\usepackage{graphicx}
\usepackage[toc,page]{appendix}
\usepackage{stfloats}
\usepackage{CJKutf8}
\usepackage[strict=true, style=british]{csquotes}
\usepackage[british]{babel}
\usepackage[backend=biber,style=alphabetic,sorting=anyt, maxbibnames=99,maxcitenames=99,maxalphanames=99,
]{biblatex}
\newcounter{oldtocdepth}

\DeclareMathAlphabet{\mathpzc}{OT1}{pzc}{m}{it}

\newtheorem{Definition}{Definition}[subsection]
\newtheorem{Theorem}[Definition]{Theorem}
\newtheorem{Lemma}[Definition]{Lemma}
\newtheorem{Proposition}[Definition]{Proposition}
\newtheorem{Corollary}[Definition]{Corollary}
\newtheorem{Example}[Definition]{Example}
\newtheorem{Remark}[Definition]{Remark}
\newtheorem{Hypothesis}[]{Hypothesis}

\newtheorem{Conjecture}[]{Conjecture}

\newtheorem*{Theorem*}{Theorem}
\newtheorem*{Example*}{Example}
\newtheorem*{Remark*}{Remark}

\DeclareMathOperator\A{\mathbf{A}}
\DeclareMathOperator\B{\mathbf{B}}
\DeclareMathOperator\C{\mathbf{C}}
\DeclareMathOperator\D{\mathbf{D}}

\DeclareMathOperator\F{\mathbf{F}\!}

\DeclareMathOperator\Q{\mathbf{Q}}
\DeclareMathOperator\R{\mathbf{R}}
\DeclareMathOperator\T{\mathbf{T}}

\DeclareMathOperator\V{\mathbf{V}}

\DeclareMathOperator\Z{\mathbf{Z}}

\DeclareMathOperator\bft{\mathbf{t}}
\DeclareMathOperator\bfu{\mathbf{u}}

\DeclareMathOperator\bfitx{\textbf{\textit{x}}}
\DeclareMathOperator\bfity{\textbf{\textit{y}}}

\DeclareMathOperator\bbF{\mathbb{F}}
\DeclareMathOperator\bbG{\mathbb{G}}
\DeclareMathOperator\bbH{\mathbb{H}}

\DeclareMathOperator\bbT{\mathbb{T}}
\DeclareMathOperator\bbX{\mathbb{X}}

\DeclareMathOperator\calD{\mathcal{D}}
\DeclareMathOperator\calE{\mathcal{E}}

\DeclareMathOperator\calG{\mathcal{G}}

\DeclareMathOperator\calQ{\mathcal{Q}}
\DeclareMathOperator\calR{\mathcal{R}}

\DeclareMathOperator\calU{\mathcal{U}}
\DeclareMathOperator\calV{\mathcal{V}}
\DeclareMathOperator\calW{\mathcal{W}}
\DeclareMathOperator\calX{\mathcal{X}}

\DeclareMathOperator\calZ{\mathcal{Z}}

\DeclareMathOperator\scrD{\mathscr{D}}

\DeclareMathOperator\scrO{\mathscr{O}}

\DeclareMathOperator\scrT{\mathscr{T}}

\DeclareMathOperator\frakI{\mathfrak{I}}

\DeclareMathOperator\fraka{\mathfrak{a}}

\DeclareMathOperator\GL{GL}
\DeclareMathOperator\GSp{GSp}
\DeclareMathOperator\GSpin{GSpin}
\DeclareMathOperator\GO{GO}

\DeclareMathOperator\Sp{Sp}

\DeclareMathOperator\End{End}
\DeclareMathOperator\Hom{Hom}

\DeclareMathOperator\ad{ad}

\DeclareMathOperator\adj{adj}

\DeclareMathOperator\alg{alg}

\DeclareMathOperator\an{an}

\DeclareMathOperator\bad{bad}

\DeclareMathOperator\charpoly{char.poly}

\DeclareMathOperator\cl{cl}

\DeclareMathOperator\cris{cris}
\DeclareMathOperator\cts{cts}

\DeclareMathOperator\Det{Det}

\DeclareMathOperator\diag{diag}

\DeclareMathOperator\Fil{Fil}

\DeclareMathOperator\Frob{Frob}
\DeclareMathOperator\Gal{Gal}

\DeclareMathOperator\Hecke{Hecke}

\DeclareMathOperator\image{image}
\DeclareMathOperator\irr{irr}

\DeclareMathOperator\Iw{Iw}

\DeclareMathOperator\length{length}

\DeclareMathOperator\ord{ord}
\DeclareMathOperator\one{\mathbbm{1}}
\DeclareMathOperator\oneanti{\breve{\one}}
\DeclareMathOperator\opp{opp}
\DeclareMathOperator\Par{par}

\DeclareMathOperator\red{red}
\DeclareMathOperator\rig{rig}
\DeclareMathOperator\Res{Res}
\DeclareMathOperator\Sbad{\mathtt{S}_{\bad}}

\DeclareMathOperator\Spa{Spa}

\DeclareMathOperator\Spec{Spec}

\DeclareMathOperator\spin{spin}

\DeclareMathOperator\tame{tame}
\DeclareMathOperator\tol{tol}

\DeclareMathOperator\trans{^{\mathtt{t}}\!}

\DeclareMathOperator\univ{univ}

\DeclareMathOperator\wt{wt}
\DeclareMathOperator\Weyl{Weyl}

\DeclareMathOperator\Sch{\textbf{\textsc{Sch}}}
\DeclareMathOperator\Sets{\textbf{\textsc{Sets}}}

\DeclareMathOperator\Ar{\textbf{\textsc{Ar}}}

\DeclareMathOperator\bfSigma{\boldsymbol{\Sigma}}

\DeclareMathOperator\bfalpha{\boldsymbol{\alpha}}
\DeclareMathOperator\bfbeta{\boldsymbol{\beta}}
\DeclareMathOperator\bfgamma{\boldsymbol{\gamma}}
\DeclareMathOperator\bfdelta{\boldsymbol{\delta}}

\DeclareMathOperator\bflambda{\boldsymbol{\lambda}}

\DeclareMathOperator\bftau{\boldsymbol{\tau}}
\DeclareMathOperator\bfupsilon{\boldsymbol{\upsilon}}

\DeclareMathOperator\llbrack{\![\![\!}
\DeclareMathOperator\rrbrack{\!]\!]}

\makeatletter
\renewcommand{\maketitle}{\bgroup\setlength{\parindent}{0pt}
\begin{flushleft}
  \LARGE{\textbf{\@title}}
  
  \vspace{4mm}
  
  \large{\textsc{\@author}}
  
  \vspace{4mm}
\end{flushleft}\egroup
}
\makeatother

\title{On adjoint Bloch--Kato Selmer groups for $\mathrm{GSp}_{2g}$}
\author{Ju-Feng Wu}
\date{}

\begin{document}

\maketitle

{\footnotesize 
\paragraph{Abstract.} We study the adjoint Bloch--Kato Selmer groups attached to a classical point in the cuspidal eigenvariety associated with $\mathrm{GSp}_{2g}$. Our strategy is based on the study of families of Galois representations on the eigenvariety, which is inspired by the book of J. Bellaiche and G. Chenevier. 
}

{\footnotesize
\paragraph{Résumé.} On étudie de group de Selmer adjoint défini par Bloch et Kato attachés à un point classique dans la variété de Hecke cuspidale pour $\mathrm{GSp}_{2g}$. Notre stratégie passe par l'étude de familles de représentations de Galois sur la variété de Hecke cuspidale, qui est inspirée par le livre de J. Bellaïche et G. Chenevier. 
}

\iffalse
{\footnotesize
\paragraph{Résumé.} Dans ce papier, on complète notre travail précédent sur l'accouplement défini sur la variété de Hecke cuspidale pour $\mathrm{GSp}_{2g}$.
Plus précisément, on justifie (conjecturalement) le nom ``fonction $L$ $p$-adique adjointe'' donné à l'objet construit dans notre travail précédent. 
Notre stratégie passe par l'étude de familles de représentations de Galois sur la variété de Hecke cuspidale pour $\mathrm{GSp}_{2g}$, et on relie la fonction $L$ $p$-adique adjointe au group de Selmer adjoint défini par Bloch et Kato. Cette stratégie est inspirée par le livre de J. Bellaïche et G. Chenevier. 
}
\fi

\tableofcontents

\section{Introduction}

In this paper, we study the adjoint Bloch--Kato Selmer groups attached to a classical point in the cuspidal eigenvariety $\calE_0$ associated with $\GSp_{2g}$. Our strategy is based on the study of families of Galois representations on $\calE_0$, which is inspired by the book of J. Bellaiche and G. Chenevier (\cite{Bellaiche-Chenevier-book}). 

\subsection*{An overview}

Fix a prime number $p$ and a positive integer $g\in \Z_{> 0}$. Let $X(\C)$ be the Siegel modular variety of a fixed tame level structure away from $p$; and let $X_{\Iw}(\C)$ be the Siegel modular variety over $X(\C)$ with an extra Iwahori level at $p$. We let $N$ be the product of primes that divide the level of $X(\C)$.

\vspace{2mm}

On $X_{\Iw}(\C)$, one can consider the \emph{overconvergent parabolic cohomology groups} $H_{\Par, \kappa}^{\tol}$. Following the formalism in \cite{Hansen-PhD}, one can use $H_{\Par, \kappa}^{\tol}$ to construct the (reduced equidimensional) \emph{cuspidal eigenvariety} $\calE_0$, parametrising finite-slope families of eigenclasses in the overconvergent parabolic cohomology groups. See \S \ref{subsection: cohomology groups} and \S \ref{subsection: Hecke operators and eigenvariety} for a review.

\vspace{2mm}

Given a point $\bfitx\in \calE_0$ whose weight is a dominant algebraic weight and whose slope is small enough, it is predicted by R. Langlands that there is a continuous Galois representation \[
    \rho_{\bfitx}: \Gal_{\Q} \xrightarrow{\rho_{\bfitx}^{\spin}} \GSpin_{2g+1}(\overline{\Q}_p) \xrightarrow{ \spin} \GL_{2^g}(\overline{\Q}_p)\footnote{ See \cite[Lecture 20]{Fulton-Harris} for the definition and properties of the spin representation $\spin: \GSpin_{2g+1} \rightarrow \GL_{2^g}$.}
\] whose characteristic polynomials of the Frobenii away from $Np$ are equal to the Hecke polynomials away from $Np$. Here, $\Gal_{\Q}$ denotes the absolute Galois group of $\Q$.  

\vspace{2mm}

Let $\ad^0 \rho_{\bfitx}^{\spin}$ be the (trace-$0$) adjoint representation of $\rho_{\bfitx}^{\spin}$ and consider the adjoint Bloch--Kato Selmer group $H_{f}^1(\Q, \ad^0 \rho_{\bfitx}^{\spin})$. We have the following conjecture of S. Bloch and K. Kato: 

\begin{Conjecture}[Bloch--Kato conjecture]\label{Conjecture: BK}
\begin{enumerate}
    \item[(i)] The order of vanishing of the adjoint $L$-function $L(\ad^0\rho_{\bfitx}^{\spin}, s)$ at $s =1$ is equal to the dimension of   $H_{f}^1(\Q, \ad^0 \rho_{\bfitx}^{\spin})$. 
    \item[(ii)] The adjoint Bloch--Kato Selmer group $H_{f}^1(\Q, \ad^0 \rho_{\bfitx}^{\spin})$ vanishes. 
\end{enumerate}
\end{Conjecture}

The aim of this paper is to show that, under certain assumptions, $H_{f}^1(\Q, \ad^0 \rho_{\bfitx}^{\spin})$ does vanish. In particular, the following natural hypotheses are assumed to achieve our goal: \begin{enumerate}
    \item[$\bullet$] Hypothesis \ref{Hyp: associated Gal. rep}: Roughly speaking, this hypothesis states that the aforementioned philosophy of Langlands holds true.
    \item[$\bullet$] Hypothesis \ref{Hyp: spin functoriality}: Roughly speaking, this hypothesis ensures that there exists a real finite extension $L$ of $\Q$ and a generic cuspidal automorphic representation $\GL_{2^g}(\A_L)$ whose associated Galois representation coincide with $\rho_{\bfitx}|_{\Gal_{L}}$, where $\A_L$ is the ring of adèles of $L$ and $\Gal_L$ is the absolute Galois group of $L$. 
    \item[$\bullet$] Hypothesis \ref{Hypothesis: minimal ramification}: This is a technical hypothesis, which ensures us to obtain a $\GSpin_{2g+1}$-valued Galois representation with coefficients in the local eigenalgebra of $\bfitx$ and that the chosen tame $\Gamma^{(p)}$ implies a particular ramification type of this Galois representation at bad primes.
\end{enumerate}

\begin{Theorem*}[Corollary \ref{Corollary: main result for BK Selmer group}]\label{Theorem: Bloch--Kato, R=T and geometry of eigenvarieties; intro}
Let $\bfitx\in \calE_0$ whose weight is a dominant algebraic weight and whose slope is small enough. Suppose the following assumptions hold:\footnote{ Since many hypotheses are assumed this theorem, examples of these assumptions are discussed in \S \ref{subsection: examples}. } \begin{enumerate}
    \item[(I)] Standard assumptions:\begin{enumerate}
        \item[$\bullet$] The point $\bfitx$ corresponds to a $p$-stabilisation of an eigenclass of tame level (see \S \ref{subsection: Galois representation for GSp} and \S \ref{subsection: families of representations on eigenvariety} for more discussion).
        \item[$\bullet$] Hypothesis \ref{Hyp: associated Gal. rep} holds so that we get a $\GSpin_{2g+1}$-valued Galois representation $\rho_{\bfitx}^{\spin}$ attached to $\bfitx$. We write $\rho_{\bfitx} := \spin \circ \rho_{\bfitx}^{\spin}$ be the associated $\GL_{2^g}$-valued Galois representation.
    \end{enumerate}
    \item[(II)] Technical assumption: Hypothesis \ref{Hypothesis: minimal ramification} hold.
    \item[(III)] Assumptions used in the strategy of \cite{Bellaiche-Chenevier-book}:\begin{enumerate}
        \item[$\bullet$]  The restriction $\rho_{\bfitx}|_{\Gal_{\Q_p}}$ admits a refinement $\bbF_{\bullet}^{\bfitx}$ that satisfies (REG) and (NCR) (see \S \ref{subsection: pseudocharacters and families} for definitions of $\bbF_{\bullet}^{\bfitx}$, (REG) and (NCR)).
        \item[$\bullet$] The restriction $\rho_{\bfitx}|_{\Gal_{\Q_p}}$ is not isomorphic to its twist by the $p$-adic cyclotomic character.
    \end{enumerate}
    \item[(IV)] Assumptions to apply \cite{Newton--Thorne}:\begin{enumerate}
        \item[$\bullet$] Hypothesis \ref{Hyp: spin functoriality} holds.
        \item[$\bullet$] The cuspidal automorphic representation $\pi_{\bfitx}$ of $\GL_{2^g}(\A_L)$ ensured by Hypothesis \ref{Hyp: spin functoriality} is regular algebraic and polarised (see, for example, \cite[\S 2.1]{BLGGT14}). 
    \item[$\bullet$] The image $\rho_{\bfitx}(\Gal_{L(\zeta_{p^{\infty}})})$ is enormous (see \cite[Definition 2.27]{Newton--Thorne}). 
    \end{enumerate}
\end{enumerate}
Then \begin{enumerate}
    \item[(i)] The adjoint Bloch--Kato Selmer group $H_f^1(\Q, \ad^0\rho_{\bfitx}^{\spin})$ associated with $\rho_{\bfitx}^{\spin}$ vanishes. 
    \item[(ii)] There is an `infinitesimal $R=T$ theorem'  locally at $\bfitx$.
\end{enumerate}
\end{Theorem*}

\begin{proof}[Strategy of the proof] 
We now summarise the strategy to achieve the statement:

\vspace{2mm}

    \noindent \textbf{Step 1.} Using the standard assumptions in the theorem and Proposition \ref{Prop: Tp and Up eigenvalues}, we construct a refined family of Galois representations $(\calE_0^{\irr}, \Det^{\univ}, \calX_{\heartsuit}^{\cl}, \{\alpha_i: i=1, ..., 2^g\}, \{F_i: i=1, ..., 2^g\})$ in Theorem \ref{Theorem: refined family on the cuspidal eigenvariety}.
    
    \vspace{2mm}
    
    \noindent \textbf{Step 2.} Following the strategy in \cite{Bellaiche-Chenevier-book} and using the assumption that $\rho_{\bfitx}$ admits a refinement $\bbF_{\bullet}^{\bfitx}$, we define global deformation problems $\scrD_{\bfitx, f}^{\spin}$ and $\scrD_{\bfitx, \bbF_{\bullet}^{\bfitx}}^{\spin}$. It is standard in Galois deformation theory that these two functors are pro-representable by complete noetherian local rings by $R_{\bfitx, f}^{\univ}$ and $R_{\bfitx, \bbF_{\bullet}^{\bfitx}}^{\spin}$ respectively.
    
    \vspace{2mm}
    
    \noindent \textbf{Step 3.} By applying a theorem of Bellaïche--Chenevier to the refined family of Galois representations in Step 1 and combining Hypothesis \ref{Hypothesis: minimal ramification} and (III) in the theorem, we deduce in Proposition \ref{Proposition: relation between R and T} the following statements: \begin{enumerate}
        \item[(i)] There exists a canonical ring homomorphism $R^{\univ}_{\bfitx, \bbF_{\bullet}^{\bfitx}}\rightarrow \bbT_{\bfitx}$, where $\bbT_{\bfitx}$ is the local eigenalgebra at $\bfitx$.
    
        \item[(ii)] If the adjoint Bloch--Kato Selmer group $H_f^1(\Q, \ad^0\rho_{\bfitx}^{\spin})$ vanishes, then the canonical map in $R^{\univ}_{\bfitx, \bbF_{\bullet}^{\bfitx}}\rightarrow \bbT_{\bfitx}$ is an isomorphism (an `infinitesimal $R=T$ theorem').
    \end{enumerate}

    \vspace{2mm}
    
    \noindent \textbf{Step 4.} To conclude the result, the assumption (IV) and \cite[Theorem 5.3]{Newton--Thorne} imply that $H_f^1(\Q, \ad^0\rho_{\bfitx}^{\spin}) = 0$. The desired assertions then follow.
\end{proof}

We close this introduction with the remark that one can also deduce the vanishing of the adjoint Bloch--Kato Selmer group without assuming (IV) in the theorem above but with two other (probably strong) assumptions (see Corollary \ref{Corollary: etale + R=T implies vanishing of adjoint BK}). Such a statement shall allow one to obtain a (conjectural) link between the \emph{$p$-adic adjoint $L$-function} $L^{\adj}$ defined in \cite{Wu-pairing} and the adjoint Bloch--Kato Selmer group (see Remark \ref{Remark: relation to the pairing}). This is, in fact, the original motivation of our study in this paper.

\subsection*{Conventions}

Throughout this paper, we fix the following: \begin{enumerate}
    \item[$\bullet$] $g\in \Z_{\geq 1}$.
    \item[$\bullet$] For any prime number $\ell$, we fix once and forever an algebraic closure $\overline{\Q}_{\ell}$ of $\Q_{\ell}$ and an algebraic isomorphism $\C_{\ell}\simeq \C$, where $\C_{\ell}$ is the $\ell$-adic completion of $\overline{\Q}_{\ell}$. We write $\Gal_{\Q_{\ell}}$ for the absolute Galois group $\Gal(\overline{\Q}_{\ell}/\Q_{\ell})$. We also fix the $\ell$-adic absolute value on $\C_{\ell}$ so that $|\ell|=\ell^{-1}$.
    \item[$\bullet$] We also fix an algebraic closure $\overline{\Q}$ of $\Q$ and embeddings $\overline{\Q}_{\ell} \hookleftarrow \overline{\Q} \hookrightarrow \C$, which is compatible with the chosen isomorphisms $\C_{\ell} \simeq \C$. We analogously write $\Gal_{\Q}$ for the absolute group $\Gal(\overline{\Q}/\Q)$ and identify $\Gal_{\Q_{\ell}}$ as a (decomposition) subgroup of $\Gal_{\Q}$. 
    \item[$\bullet$] We fix an odd prime number $p\in \Z_{> 0}$. 
    \item[$\bullet$] For $n\in \Z_{\geq 1}$ and any set $R$, we denote by $M_n(R)$ the set of $n$ by $n$ matrices with coefficients in $R$.
    \item[$\bullet$] The transpose of a matrix $\bfalpha$ is denoted by $\trans\bfalpha$.
    \item[$\bullet$] For any $n\in \Z_{\geq 1}$, we denote by $\one_n$ the $n\times n$ identity matrix and denote by $\oneanti_n$ the $n\times n$ anti-diagonal matrix whose non-zero entries are $1$; \emph{i.e.,} $$\one_n=\begin{pmatrix} 1& & \\ & \ddots & \\ & &1\end{pmatrix}\quad\text{ and }\quad\oneanti_n=\begin{pmatrix} & & 1\\ & \iddots & \\ 1 & &\end{pmatrix}.$$
\end{enumerate}

\paragraph{Acknowledgement.} The author is in debt to Patrick Allen for answering the author's many questions about Galois deformation theory as well as pointing out the reference \cite{Newton--Thorne}, which leads to an improvement of this work. To him, the author is grateful.
The author also thanks James Newton for pointing out mistakes in the previous versions of the paper.
The author wishes to thank Arno Kret for his patience in answering the author's questions regarding his work with Shin. Many thanks also go to Giovanni Rosso for his constant help and encouragement. Many of the materials presented here benefited from countless discussions with him. The author is also grateful to him for pointing out mistakes in an early draft and for his help with the French translation of the abstract. 
Many thanks are also owed to Luca Mastella for carefully reading the early draft of this paper and providing helpful suggestions.  
Last but not least, the author thanks the anonymous referee for useful suggestions, which led to the improvement of the exposition of the article. 
\section{Preliminaries}\label{section: preliminaries}

In this section, we recall some preliminaries. In particular, after setting up the notations in \S \ref{subsection: algebraic and p-adic groups} and recalling the Siegel modular varieties in \S \ref{subsection: Siegel variety}, we briefly review the construction of the overconvergent parabolic cohomology groups and the construction of the cuspidal eigenvariety in \S \ref{subsection: cohomology groups} and \S \ref{subsection: Hecke operators and eigenvariety}.

\subsection{Algebraic and \texorpdfstring{$p$}{p}-adic groups}\label{subsection: algebraic and p-adic groups}

Let $\V=\V_{\Z}$ be the free $\Z$-module $\Z^{2g}$ of rank $2g$. By viewing elements in $\V$ as column vectors, we equip $\V$ with the symplectic pairing \begin{equation}\label{eq: symplectic pairing for V}
    \V \times \V \rightarrow \Z, \quad (v, v')\mapsto \trans v\begin{pmatrix} & -\oneanti_g\\ \oneanti_g\end{pmatrix} v'.
\end{equation} The algebraic group $\GSp_{2g}$ (over $\Z$) is then defined to be the group that preserves this symplectic pairing up to units. More precisely, for any ring $R$, \[
    \GSp_{2g}(R) = \left\{\bfgamma\in \GL_{2g}(R): \trans\bfgamma \begin{pmatrix} & -\oneanti_g\\ \oneanti_g\end{pmatrix} \bfgamma = \varsigma(\bfgamma) \begin{pmatrix} & -\oneanti_g\\ \oneanti_g\end{pmatrix} \text{ for some } \varsigma(\bfgamma)\in R^\times\right\}.
\]
Equivalently, for any $\bfgamma=\begin{pmatrix}\bfgamma_a & \bfgamma_b\\ \bfgamma_c & \bfgamma_d\end{pmatrix}\in \GL_{2g}$, $\bfgamma\in \GSp_{2g}$ if and only if
\[
    \trans\bfgamma_a\oneanti_g\bfgamma_c=\trans\bfgamma_c\oneanti_g\bfgamma_a, \quad \trans\bfgamma_b\oneanti_g\bfgamma_d=\trans\bfgamma_d\oneanti_g\bfgamma_b, \text{ and }\trans\bfgamma_a\oneanti_g\bfgamma_d-\trans\bfgamma_c\oneanti_g\bfgamma_b=\varsigma(\bfgamma)\oneanti_g
\] 
for some $\varsigma(\bfgamma)\in \bbG_m$. One can easily check that $\GSp_{2g}$ is stable under transpose. Thus, the above conditions are also equivalent to
\[
    \bfgamma_a\oneanti_g\trans\bfgamma_b=\bfgamma_b\oneanti_g\trans\bfgamma_a, \quad \bfgamma_c\oneanti_g\trans\bfgamma_d=\bfgamma_d\oneanti_g\trans\bfgamma_c,\quad \textrm{and }\bfgamma_a\oneanti_g\trans\bfgamma_d-\bfgamma_b\oneanti_g\trans\bfgamma_c=\varsigma(\bfgamma)\oneanti_g
\]
for some $\varsigma(\bfgamma)\in \bbG_m$.

\vspace{2mm}

We shall be also considering the following algebraic and $p$-adic subgroups of $\GL_g$ and $\GSp_{2g}$: \begin{enumerate}
    \item[$\bullet$] We consider the upper triangular Borel subgroups \begin{align*}
        B_{\GL_g} & := \text{the Borel subgroup of upper triangular matrices in $\GL_g$}\\
        B_{\GSp_{2g}} & := \text{the Borel subgroup of upper triangular matrices in $\GSp_{2g}$.}
    \end{align*} The reason why we are able to consider the upper triangular Borel subgroup for $\GSp_{2g}$ is because of the choice of the pairing in (\ref{eq: symplectic pairing for V}).
    
    \item[$\bullet$] The corresponding unipotent radicals are \begin{align*}
        U_{\GL_g} & := \text{ the upper triangular $g\times g$ matrices whose diagonal entries are all $1$}\\
        U_{\GSp_{2g}} & := \text{ the upper triangular $2g\times 2g$ matrices in $\GSp_{2g}$ whose diagonal entries are all $1$}.
    \end{align*} Consequently, the maximal tori for both algebraic groups are the tori of diagonal matrices. The Levi decomposition then yields \[
        B_{\GL_g}=U_{\GL_g}T_{\GL_g}\text{ and }B_{\GSp_{2g}}=U_{\GSp_{2g}}T_{\GSp_{2g}}.
    \]
    Moreover, we denote by $U_{\GL_g}^{\opp}$ and $U_{\GSp_{2g}}^{\opp}$ the opposite unipotent radical of $U_{\GL_g}$ and $U_{\GSp_{2g}}$ respectively.
    
    \item[$\bullet$] To simplify the notation, for any $s\in \Z_{\geq 0}$, we write \begin{align*}
        T_{\GL_g, s} & := \left\{\begin{array}{ll}
            T_{\GL_g}(\Z_p), & s=0 \\
            \ker(T_{\GL_g}(\Z_p)\rightarrow T_{\GL_g}(\Z/p^s\Z)), & s>0
        \end{array}\right.\\
        T_{\GSp_{2g}, s} & := \left\{\begin{array}{ll}
            T_{\GSp_{2g}}(\Z_p), & s=0 \\
            \ker(T_{\GSp_{2g}}(\Z_p)\rightarrow T_{\GSp_{2g}}(\Z/p^s\Z)), & s>0
        \end{array}\right.\\
        U_{\GL_g, s} & := \left\{\begin{array}{ll}
            U_{\GL_g}(\Z_p), & s=0 \\
            \ker(U_{\GL_g}(\Z_p)\rightarrow U_{\GL_g}(\Z/p^s\Z)), & s>0
        \end{array}\right.\\
        U_{\GSp_{2g}, s} & := \left\{\begin{array}{ll}
            U_{\GSp_{2g}}(\Z_p), & s=0 \\
            \ker(U_{\GSp_{2g}}(\Z_p)\rightarrow U_{\GSp_{2g}}(\Z/p^s\Z)), & s>0
        \end{array}\right. .
    \end{align*} The above maps are all reduction maps. 
    
    \item[$\bullet$] The Iwahori subgroups are \begin{align*}
        \Iw_{\GL_g} & : = \text{ the preimage of $B_{\GL_g}(\F_p)$ under the reduction map $\GL_g(\Z_p)\rightarrow \GL_g(\F_p)$}\\
        \Iw_{\GSp_{2g}} & := \text{ the preimage of $B_{\GSp_{2g}}(\F_p)$ under the reduction map $\GSp_{2g}(\Z_p)\rightarrow \GSp_{2g}(\F_p)$}.
    \end{align*} We have Iwahori decompositions for $\Iw_{\GL_g}$ and $\Iw_{\GSp_{2g}}$ 
    \[
    \Iw_{\GL_g} = U_{\GL_g, 1}^{\opp} T_{\GL_g, 0}U_{\GL_g, 0} \quad \text{ and }\quad \Iw_{\GSp_{2g}} = U_{\GSp_{2g}, 1}^{\opp} T_{\GSp_{2g}, 0} U_{\GSp_{2g}, 0}.
    \]
\end{enumerate}

For later purposes, we also recall the Weyl groups of $\GSp_{2g}$ and $H:=\GL_g\times \bbG_m$ from \cite[Chapter VI, \S 5]{Faltings-Chai}. Here, we view $H$ as an algebraic subgroup of $\GSp_{2g}$ via the embedding \[
    H = \GL_g \times \bbG_m \hookrightarrow \GSp_{2g}, \quad (\bfgamma, \bfupsilon)\mapsto \begin{pmatrix} \bfgamma & \\ & \bfupsilon \oneanti_g \trans\bfgamma^{-1}\oneanti_g\end{pmatrix}.
\]

Consider the character group $\bbX = \Hom(T_{\GSp_{2g}}, \bbG_m)$. We have the following isomorphism \[
    \Z^{g+1} \xrightarrow{\sim} \bbX, \quad (k_1, ..., k_g; k_0)\mapsto \left(\diag(\bftau_1, ..., \bftau_g, \bftau_0\bftau_g^{-1}, ..., \bftau_0\bftau_1^{-1})\mapsto \prod_{i=0}^{g} \bftau_i^{k_i}\right).
\] Let $x_1, ..., x_g, x_0$ be the basis of $\bbX$ that corresponds to the standard basis on $\Z^{g+1}$. Note that $\bbX$ can also be viewed as the character group of the maximal torus $T_H=T_{\GL_g}\times \bbG_m$ of $H$ via the isomorphisms $T_{\GSp_{2g}}\simeq \bbG_m^{g+1}\simeq T_{\GL_g}\times \bbG_m = T_H$.

\vspace{2mm}

Under the above choices of the maximal tori, we can describe the root systems of $\GSp_{2g}$ and $H$ explicitly \begin{align*}
    \Phi_{\GSp_{2g}} & = \{\pm(x_i-x_j),\,\, \pm(x_i+x_j-x_0),\,\, \pm(2x_t-x_0): 1\leq i<j\leq g, 1\leq t\leq g\}\\
    \Phi_H & = \{\pm(x_i-x_j),\,\, \pm x_g,\,\, \pm x_0: 1\leq i<j\leq g\}.
\end{align*} Moreover, the choices of the Borel subgroups yield the description of the positive roots \begin{align*}
    \Phi_{\GSp_{2g}}^+ & = \{x_i-x_j,\,\, x_i+x_j-x_0,\,\, 2x_t-x_0: 1\leq i<j\leq g, 1\leq t\leq g\}\\
    \Phi_H^+ & = \{x_i-x_j: 1\leq i<j\leq g\}(=\Phi_H\cap \Phi_{\GSp_{2g}}^+).
\end{align*}

The Weyl groups of $\GSp_{2g}$ and $H$ are defined as \[
    \Weyl_{\GSp_{2g}} := N_{\GSp_{2g}}(T_{\GSp_{2g}})/T_{\GSp_{2g}} \quad \text{ and }\quad \Weyl_H := N_H(T_H)/T_H,
\]
where $N_{\GSp_{2g}}(T_{\GSp_{2g}})$ (resp. $N_{H}(T_H)$) is the group of normalisers of $T_{\GSp_{2g}}$ (resp. $T_H$) in $\GSp_{2g}$ (resp. $H$). They can also be described explicitly as follows. \begin{enumerate}
    \item[$\bullet$] We can identify $\Weyl_{\GSp_{2g}}$ with $\bfSigma_g \ltimes (\Z/2\Z)^g$, where $\bfSigma_g$ denotes the permutation group on $g$ letters. For any $\bftau = \diag(\bftau_1, ..., \bftau_g, \bftau_0\bftau_g^{-1}, ..., \bftau_0\bftau_1^{-1})\in T_{\GSp_{2g}}$, the actions of $\bfSigma_g$ and $(\Z/2\Z)^{g}$ are given as \begin{enumerate}
        \item[(i)] $\bfSigma_g$ permutes $\bftau_1, ..., \bftau_g$,
        \item[(ii)] the element $(\underbrace{0, ..., 0}_{i-1}, 1, 0, ..., 0)\in (\Z/2\Z)^{g}$ maps $\bftau$ to \[
            \diag(\bftau_1, ..., \bftau_{i-1}, \bftau_0\bftau_i^{-1}, \bftau_{i+1}, ...,\bftau_g, \bftau_0\bftau_g^{-1}, ..., \bftau_0\bftau_{i+1}^{-1}, \bftau_i, \bftau_0\bftau_{i-1}^{-1}, ..., \bftau_0\bftau_{1}^{-1}).
        \]
    \end{enumerate}
    \item[$\bullet$] We can identify $\Weyl_H$ with $\bfSigma_g$, whose action on $T_H$ is defined as the action of $\bfSigma_g$ on $T_{\GSp_{2g}}$.
\end{enumerate} The actions of the Weyl groups on the maximal tori then induce actions on the root systems $\Phi_{\GSp_{2g}}$ and $\Phi_H$. Following \cite[Chapter VI, \S 5]{Faltings-Chai}, let \[
    \Weyl^H : = \{ x\in \Weyl_{\GSp_{2g}}: x(\Phi_{\GSp_{2g}}^+)\supset \Phi_H^+\} \subset \Weyl_{\GSp_{2g}}. 
\] The subset $\Weyl^H$ consequently gives a system of representatives of the quotient $\Weyl_H\backslash \Weyl_{\GSp_{2g}}$.

\subsection{The Siegel modular varieties}\label{subsection: Siegel variety} 
Let $\Gamma^{(p)} \subset \GSp_{2g}(\widehat{\Z})$ be a neat open compact subgroup such that $\Gamma^{(p)} = \prod_{\ell: \text{ prime}} \Gamma^{(p)}_{\ell}$, where each $\Gamma^{(p)}_{\ell}$ is an open compact subgroup of $\GSp_{2g}(\Z_{\ell})$. Let $\Sbad := \{\ell \text{ prime number}: \Gamma_{\ell}^{(p)} \subsetneq \GSp_{2g}(\Z_{\ell})\} \cup \{p\}$. We shall assume this union is a disjoint union and write $N:= \prod_{\ell\in \Sbad \smallsetminus \{p\}}\ell$ ( and so $p\nmid N$).

\vspace{2mm}

Fix a primitive $N$-th roots of unity $\zeta_N\in \overline{\Q}\subset \overline{\Q}_p$. Let $\Sch_{\Z_p[\zeta_N]}$ be the category of locally noetherian schemes over $\Z_p[\zeta_N]$. Consider the functor \[
    \Sch_{\Z_p[\zeta_N]} \rightarrow \Sets, \quad S\mapsto \left\{(A_{/S}, \lambda, \psi_N): \begin{array}{l}
        A \text{ is a principally polarised abelian scheme over }S   \\
        \lambda \text{ is a principal polarisation on }A\\
        \psi_N\text{ is a level structure defined by }\Gamma^{(p)}
    \end{array}\right\}/\simeq .
\] 
Assume that $\Gamma^{(p)}$ is chosen so that the above functor is representable by a scheme $X_{\Z_p[\zeta_N]}$. Denote by $X=X_{\C_p}$ the base change of $X_{\Z_p[\zeta_N]}$ to $\C_p$.

\begin{Example}\label{Example: tame level}
\normalfont Suppose $\Gamma^{(p)} = \Gamma(N) := \ker(\GSp_{2g}(\widehat{\Z})\rightarrow \GSp_{2g}(\Z/N\Z))$ for $N$ large enough, then $\Gamma(N)$ defines the level structure asking for symplectic isomorphisms,   \[
    \psi_N: A[N] \xrightarrow{\sim} (\Z/N\Z)^{2g},
\]\emph{i.e.}, isomorphisms that preserve symplectic pairings on both sides up to units, where we consider the  Weil pairing on the left-hand side and the symplectic pairing induced by (\ref{eq: symplectic pairing for V}) on the right-hand side. \qed
\end{Example}

Fix a primitive $p$-th root of unity $\zeta_p\in \overline{\Q}\subset \overline{\Q}_p$, we also consider the scheme $X_{\Iw, \Q_p[\zeta_N, \zeta_p]}$, parametrising tuples $$(A, \lambda, \psi_N, \Fil_{\bullet}),$$ where $(A, \lambda, \psi_N)\in X_{\Q_p[\zeta_N, \zeta_p]}:=X_{\Z_p[\zeta_N]}\times_{\Z_p[\zeta_N]}\Spec \Q_p[\zeta_p, \zeta_N]$ and $\Fil_{\bullet}$ is a full filtration of $A[p]$ such that $\Fil_{\bullet}^{\perp} = \Fil_{2g-\bullet}$ (with respect to the Weil pairing).  Similarly, we write $X_{\Iw}=X_{\Iw, \C_p}$ the base change of $X_{\Iw, \Q_p[\zeta_N, \zeta_p]}$ to $\C_p$. Obviously, we have the natural forgetful morphism \[
    X_{\Iw} \rightarrow X, \quad (A, \lambda, \psi_N, \Fil_{\bullet}) \mapsto (A, \lambda, \psi_N).
\] This morphism is obviously an \'etale morphism.

\vspace{2mm}

With respect to the fixed isomorphism $\C\simeq \C_p$ in the beginning, we can consider the $\C$-points $X_{\Iw}(\C)$. The space $X_{\Iw}(\C)$ can then be identified with the locally symmetric space \[
    X_{\Iw}(\C) = \GSp_{2g}(\Q)\backslash \GSp_{2g}(\A_f)\times \bbH_g/\Iw_{\GSp_{2g}}\Gamma^{(p)},
\]
where $\bbH_g$ is the (disjoint) union of the Siegel upper- and lower-half plane and $\A_f$ is the ring of finite ad\`eles of $\Q$. 
It is well-known that the dimension of the Siegel modular variety $X_{\Iw}$ (as well as $X$) is $n_0= g(g+1)/2$.

\vspace{2mm}

Let $M$ be a left $\GSp_{2g}(\Z_p)$-module (over some commutative ring). Thus, $M$ also admits a left action of $\Iw_{\GSp_{2g}}$ via restriction. By equipping $M$ with a trivial action by $\Gamma^{(p)}$, the module $M$ naturally defines a local system on $X(\C)$ and  $X_{\Iw}(\C)$ as explained in \cite[\S 2.2]{Ash-Stevens} (see also \cite[\S 2.1]{Hansen-PhD}). One can then consider the (Betti) cohomology groups $H^t(X(\C), M)$ and  $H^t(X_{\Iw}(\C), M)$ (resp. compactly supported cohomology groups $H^t_c(X(\C), M)$ and $H^t_c(X_{\Iw}(\C), M)$) with coefficients in $M$. Then, there are natural morphisms \begin{align}
    \Lambda_p: H^t(X(\C), M) & \rightarrow H^t(X_{\Iw}(\C), M) \label{eq: p-old morphism}\\
    \Lambda_p: H_c^t(X(\C), M) & \rightarrow H_c^t(X_{\Iw}(\C), M) \label{eq: p-old morphism for compactly supported cohomology}
\end{align} 
induced by the forgetful morphism.

\subsection{The overconvergent parabolic cohomology groups}\label{subsection: cohomology groups}
Define \[
    \T_0 = \left\{(\bfgamma, \bfupsilon)\in \Iw_{\GL_g} \times M_g(p\Z_p) : \trans\bfgamma\oneanti_g\bfupsilon= \trans\bfupsilon\oneanti_g\bfgamma\right\}.
\] Elements in $\T_0$ can be viewed as the left $(2g\times g)$-columns of matrices in $\Iw_{\GSp_{2g}}$ as explained in \cite[\S 2.2]{Wu-pairing}. Then $\T_0$ admits a right action of $B_{\GL_g, 0} $ given by the right multiplication and a left action of $\Xi : = \begin{pmatrix} \Iw_{\GL_g} & M_g(\Z_p)\\ M_g(p\Z_p) & M_g(\Z_p)\end{pmatrix}\cap \GSp_{2g}(\Q_p)$ by the \textit{left multiplication}. Moreover, $\T_0$ admits a special subset \[
    \T_{00} : = \{(\bfgamma, \bfupsilon)\in \T_0: \bfgamma\in U_{\GL_g, 1}^{\opp}\},
\] which can be identified with $U_{\GSp_{2g}, 1}^{\opp}$ via \[
    \T_{00} \xrightarrow{\sim} U_{\GSp_{2g}, 1}^{\opp}, \quad (\bfgamma, \bfupsilon)\mapsto \begin{pmatrix} \bfgamma & \\ \bfupsilon & \oneanti_g \trans\bfgamma^{-1}\oneanti_g\end{pmatrix}.
\]

For any affinoid $\Q_p$-algebra $R$ and any $p$-adic weight (\emph{i.e.}, continuous character) \(
    \kappa: T_{\GL_g, 0} \rightarrow R^\times
\) and any $s\in \Z_{>0}$, we consider the $s$-locally analytic functions \[
    A_{\kappa}^s(\T_0, R) := \left\{\phi: \T_0\rightarrow R : \begin{array}{l}
        \phi(\bfgamma\bfbeta, \bfupsilon\bfbeta) = \kappa(\bfbeta)\phi(\bfgamma, \bfupsilon)\,\,\forall ((\bfgamma, \bfupsilon), \bfbeta)\in \T_0\times B_{\GL_g, 0}  \\
        \phi|_{\T_{00}} \text{ is $s$-locally analytic}
    \end{array}\right\}.
\] Here, we extend $\kappa$ to a function on $B_{\GL_g, 0}$ by setting $\kappa|_{U_{\GL_g, 0}} = 1$ and the `$s$-locally analytic' condition is in the sense of \cite[\S 2 Definition]{Hansen-PhD} (after identifying $\T_{00}$ with $U_{\GSp_{2g}, 1}^{\opp}$). One sees immediately that we have a natural inclusion $A_{\kappa}^s(\T_0, R)\subset A_{\kappa}^{s+1}(\T_0, R)$.

\vspace{2mm}

The $s$-locally analytic distributions are then defined to be \[
    D_{\kappa}^s(\T_0, R) : = \Hom_{R}^{\cts}(A_{\kappa}^{s}(\T_0, R), R).
\] The natural inclusion $A_{\kappa}^s(\T_0, R)\subset A_{\kappa}^{s+1}(\T_0, R)$ then yields a natural projection $D_{\kappa}^{s+1}(\T_0, R)\rightarrow D_{\kappa}^s(\T_0, R)$. Consequently, we define \begin{align*}
    & A_{\kappa}^{\dagger}(\T_0, R) := \varinjlim_{s} A_{\kappa}^s(\T_0, R)\\
    & D_{\kappa}^{\dagger} (\T_0, R) := \varprojlim_{s} D_{\kappa}^s(\T_0, R).
\end{align*} We call elements of these two modules \textbf{\textit{overconvergent functions}} and \textbf{\textit{overconvergent distributions}} respectively. It is also obvious that $D_{\kappa}^{\dagger}(\T_0, R)$ is the continuous dual of $A_{\kappa}^{\dagger}(\T_0, R)$.

\vspace{2mm}

Observe that $\Iw_{\GSp_{2g}}\subset \Xi$, thus $D_{\kappa}^{\dagger}(\T_0, R)$ is naturally a left $\Iw_{\GSp_{2g}}$-module. Equip it with a trivial action by $\Gamma^{(p)}$, we can consequently consider the cohomology groups (resp. compactly supported cohomology groups) $H^t(X_{\Iw}(\C), D_{\kappa}^{\dagger}(\T_0, R))$ (resp. $H_c^t(X_{\Iw}(\C), D_{\kappa}^{\dagger}(\T_0, R))$) for any $0\leq t\leq 2n_0$. The \textbf{\textit{overconvergent parabolic cohomology group}} is then defined to be \[
    H_{\Par}^t(X_{\Iw}(\C), D_{\kappa}^{\dagger}(\T_0, R)):= \image\left(H_c^t(X_{\Iw}(\C), D_{\kappa}^{\dagger}(\T_0, R)) \rightarrow H^t(X_{\Iw}(\C), D_{\kappa}^{\dagger}(\T_0, R))\right),
\] where the map is the natural map from the compactly supported cohomology group to the cohomology group. In what follows, we will be considering the total overconvergent parabolic cohomology group \[
    H_{\Par, \kappa}^{\tol} := \oplus_{t=0}^{2n_0} H_{\Par}^t(X_{\Iw}(\C), D_{\kappa}^{\dagger}(\T_0, R)).
\]

\subsection{Hecke operators and the (reduced equidimensional) cuspidal eigenvariety}\label{subsection: Hecke operators and eigenvariety}
Let $\ell$ be a prime number that does not divide $pN$. We consider the set of double cosets \[\Upsilon_{\ell}:=\{[\GSp_{2g}(\Z_{\ell})\bfdelta\GSp_{2g}(\Z_{\ell})]: \bfdelta\in \GSp_{2g}(\Q_q)\cap M_{2g}(\Z_q)\}.
\]
For any fixed $\bfdelta$, we have the coset decomposition \[\GSp_{2g}(\Z_{\ell})\bfdelta\GSp_{2g}(\Z_{\ell})=\sqcup_{j} \bfdelta_j\GSp_{2g}(\Z_{\ell})
\] for finitely many $\bfdelta_j\in \GSp_{2g}(\Q_{\ell})\cap M_{2g}(\Z_{\ell})$. By letting $\bfdelta_j$'s act trivially on $D_{\kappa}^{\dagger}(\T_0, R)$, we have a left action of the double coset $[\GSp_{2g}(\Z_{\ell})\bfdelta\GSp_{2g}(\Z_{\ell})]$ on the cochain complex $C^{\bullet}_{\kappa}$ (resp. $C_{c, \kappa}^{\bullet}$) that computes the cohomology groups $H^{t}(X_{\Iw}(\C), D_{\kappa}^{\dagger}(\T_0, R))$ (resp. the compactly supported cohomology groups $H_c^t(X_{\Iw}(\C), D_{\kappa}^{\dagger}(\T_0, R))$) by $$[\GSp_{2g}(\Z_q)\bfdelta\GSp_{2g}(\Z_q)]\cdot \sigma = \sum_{j}\bfdelta_j\cdot \sigma$$ for any $\sigma\in C^{\bullet}_{\kappa}$ (resp. $C_{c, \kappa}^{\bullet}$). Then the Hecke algebra at $\ell$ (over $\Z_{p}$) is defined to be $\bbT_{\ell}=\bbT_{{\ell}, \Z_p}=\Z_p[\Upsilon_{\ell}]$. Consequently, the \textbf{\textit{unramified Hecke algebra}} is \[
    \bbT^p:= \otimes_{\ell\nmid pN}\bbT_{\ell}.
\]

We specify out a special element $\bft_{\ell, 0} = \diag(\one_g, \ell\one_g)\in \GSp_{2g}(\Q_{\ell})\cap M_{2g}(\Z_{\ell})$. For any $x\in \Weyl_{\GSp_{2g}}$, denote by $T_{\ell, 0}^x$ the Hecke operator defined by the double coset $[\GSp_{2g}(\Z_{\ell}) (x\cdot\bft_{\ell, 0}) \GSp_{2g}(\Z_{\ell})]$. Following \cite[\S 3]{Genestier-Tilouine}, we define the \textit{\textbf{Hecke polynomial at $\ell$}} to be \begin{equation}\label{eq: Hecke poly at l}
    P_{\Hecke, \ell}(Y) : = \prod_{x\in \Weyl^H}(Y-T_{\ell, 0}^x)\in \bbT_{\ell}[Y].
\end{equation} One sees immediately that this is a polynomial of degree $2^g$.

\vspace{2mm}

For Hecke operators at $p$, consider matrices
\[
    \bfu_{p, i} : = \left\{\begin{array}{ll}
        \begin{pmatrix}\one_g\\ & p\one_g\end{pmatrix}, & i=0 \\ \\
        \begin{pmatrix} \one_{g-i} \\ & p\one_{i}\\ & & p\one_{i}\\ & & & p^2\one_{g-i}\end{pmatrix}, & 1\leq i\leq g-1
    \end{array}\right. .
\] For any $(\bfgamma, \bfupsilon)\in \T_0$, write $(\bfgamma, \bfupsilon) = (\bfgamma_0, \bfupsilon_0)\bfbeta$ for some $\bfbeta\in B_{\GL_g, 0}^+$ such that $\bfgamma_0\in U_{\GL_g, 1}^{\opp}$. Then, the left action of $\bfu_{p,i}$ on $\T_0$ is defined by the formula \[
    \bfu_{p,i}\cdot (\bfgamma, \bfupsilon) = (\bfu_{p, i}^{\square}\bfgamma_0\bfu_{p, i}^{\square, -1}, \bfu_{p,i}^{\blacksquare}\bfupsilon_{0}\bfu_{p, i}^{\square, -1})\bfbeta,
\] where we write 
\[
    \bfu_{p, i} = \begin{pmatrix}\bfu_{p, i}^{\square} & \\ & \bfu_{p,i}^{\blacksquare}\end{pmatrix}.
\]
On the other hand, we also have a coset decomposition of $\Iw_{\GSp_{2g}}\bfu_{p,i}\Iw_{\GSp_{2g}}$, given by 
\[
\Iw_{\GSp_{2g}}\bfu_{p, i}\Iw_{\GSp_{2g}}=\sqcup_{j}\bfdelta_{i, j}\Iw_{\GSp_{2g}}
\] 
for some $\bfdelta_{i, j}\in \GSp_{2g}(\Q_p)\cap M_{2g}(\Z_p)$; in particular, $\bfdelta_{i,j} = \bflambda_{i,j}\bfu_{p, i}$ for some $\bflambda_{i,j}\in \Iw_{\GSp_{2g}}$. Hence, we have the action
\[
    [\Iw_{\GSp_{2g}} \bfu_{p,i} \Iw_{\GSp_{2g}}] \cdot \sigma : = \sum_{j}\bfdelta_{i,j}\cdot \sigma = \sum_{j} \bflambda_{i,j}\cdot \left(\bfu_{p, i}\cdot \sigma\right)
\]
for any $\sigma\in C^{\bullet}_{\kappa}$ (resp. $C_{c, \kappa}^{\bullet}$). We denote by $U_{p,i}$ the Hecke operator defined by the double coset $[\Iw_{\GSp_{2g}} \bfu_{p,i} \Iw_{\GSp_{2g}}]$. Similarly, for any $x\in \Weyl_{\GSp_{2g}}$, we denote by $U_{p,i}^{x}$ the Hecke operator defined by the double coset $[\Iw_{\GSp_{2g}} (x\cdot \bfu_{p, i}) \Iw_{\GSp_{2g}}]$, whose action is similarly defined as above. Then, the Hecke algebra at $p$ is defined to be $\bbT_{p} = \bbT_{p, \Z_p} = \Z_p[U_{p,i}^x: i=0, 1, ..., g-1, w\in \Weyl_{\GSp_{2g}}]$. Consequently, the \textit{\textbf{(universal) Hecke algebra}} is defined to be\[
    \bbT := \bbT^p\otimes_{\Z_p} \bbT_p.
\]

\vspace{2mm}

There is a special Hecke operator $U_p\in \bbT_p$ defined to be \[
    U_p : = \prod_{i=0}^{g-1}U_{p, i}.
\] Combining the discussions in \cite[\S 2.2]{Hansen-PhD} and \cite[\S 3.2]{Johansson-Newton},\footnote{ Let us explain this implication in more details. In \cite{Hansen-PhD} the operator $U_p$ acts compactly on the chain complexes that computes the homology groups with coefficients in $A_{\kappa}^s(\T_0, R)$ for any $s\in \Z_{>0}$. On the other hand, the authors of \cite{Johansson-Newton} used a different formalism that allows them to deduce the compactness of the operator $U_p$ on the cochain complexes that compute cohomology groups with coefficients in `$\calD_{\kappa}^r$'. The modules $\calD_{\kappa}^r$ are obtained by considering the completion on $R\llbrack U_{\GSp_{2g}, 1}^{\opp}\rrbrack$ with respect to an `$r$-norm'. Such a module is not the module of $s$-locally analytic distributions considered in \cite{Hansen-PhD} and here. However, this difference disappears after taking limit, \emph{i.e.}, $\varprojlim_{r}\calD_{\kappa}^r = D_{\kappa}^{\dagger}(\T_0, R)$. We should also caution the reader that the $p$-adic weight $\kappa$ and $R$ considered in \cite{Johansson-Newton} are well-chosen so that their formalism could be applied. We omitted this subtlety in the above discussion just to provide an idea.} the operator $U_p$ defines a compact operator on $C_{\kappa}^{\bullet}$ (resp., $C_{c, \kappa}^{\bullet}$). Consequently, we consider the slope decomposition on $C_{\kappa}^{\bullet}$ (resp., $C_{c, \kappa}^{\bullet}$) with respect to the action of $U_p$, which allows us to consider the finite slope cohomology groups (resp., compactly supported cohomology groups).

\vspace{2mm}

%Consequently, the (universal) Hecke algebra in our consideration is the $\Z_p$-algebra \[    \bbT := \bbT_p \otimes_{\Z_p} \bbT^p.\] In what follows, for any left $\Iw_{\GSp_{2g}}^+$-module $M$ that is equipped with an action of $\bbT$, we say a cohomology class in $H^t(X_{\Iw^+}, M)$ is a \textbf{\textit{Hecke-eigen class}} if and only if it is a simultaneous eigenvector of the Hecke algebra $\bbT$. \vspace{1mm}

Let $\calW = \Spa(\Z_p\llbrack T_{\GL_g,0}\rrbrack, \Z_p\llbrack T_{\GL_g, 0}\rrbrack)_{\eta}^{\an}$ be our weight space, where the superscript `$\bullet^{\an}$' means that we are taking the analytic points of the adic space and the subscript `$\bullet_\eta$' means that we are considering the generic fibre of the adic space. The slope decomposition on the cochain complexes $C_{\kappa}^{\bullet}$ then defines a Fredholm surface $\calZ$ over $\calW$. As the natural map $C_{c, \kappa}^{\bullet} \rightarrow C_{\kappa}^{\bullet}$ is Hecke-equivariant, the finite-slope cohomology groups and finite-slope compactly supported cohomology groups define finite-slope parabolic cohomology groups $H_{\Par, \kappa}^{\tol, <h}$ (see \cite[\S 3.3]{Wu-pairing}).

\vspace{2mm}

For any slope datum $(\calU, h)$ (see \cite[\S 3.1]{Hansen-PhD}; in particular, $\calU\subset \calW$), denote by $\kappa_{\calU}$ the universal weight on $\calU$ and define 
\[
    \bbT_{\Par, \calU}^{\red, h}:= \image\left(\bbT\rightarrow \End_{\scrO_{\calW}(\calU)}\left(H_{\Par, \kappa_{\calU}}^{\tol, \leq h}\right)\right)^{\red},
\] where the superscript `$\bullet^{\red}$' stands for the maximal reduced quotient of the corresponding ring. The algebras $\bbT_{\Par, \calU}^{\red, h}$ then glue together to a coherent sheaf of $\scrO_{\calZ}$-algebras, denoted $\scrT_{\Par}^{\red}$. The reduced cuspidal eigenvariety is then defined to be \[
    \calE_0^{\red} := \Spa_{\calZ}(\scrT_{\Par}^{\red}, \scrT_{\Par}^{\red, \circ}),
\]
where the sheaf of integral elements $\scrT_{\Par}^{\red, \circ}$ is guaranteed by \cite[Lemma A.3]{Johansson-Newton}. We finally define the \textbf{\textit{(reduced equidimensional) cuspidal eigenvariety}}
\[
    \calE_0 := \text{the equidimensional locus of }\calE_0^{\red}
\]The natural map \[
    \wt: \calE_0 \rightarrow \calW
\]
is called the \textbf{\textit{weight map}}. 

\begin{Remark}\label{Remark: reduced cuspidal eigenvariety}
\normalfont If we work with the \emph{strict} Iwahori level as in \cite{Wu-pairing}, then $\calE_0$ is the reduced and equidimensional part of the $p\neq 0$ locus of the cuspidal eigenvariety considered in \emph{loc. cit.}. We focus on the reduced cuspidal eigenvariety due to later purposes on families of Galois representations. 
\end{Remark}
\section{Families of Galois representations}\label{section: Galois representations}
In this section, we study families of Galois representations on the reduced equidimensional cuspidal eigenvariety $\calE_0$. We shall first recall several formalisms about families of Galois representations from \cite{Bellaiche-Chenevier-book}. Our main results concerning the Bloch--Kato conjecture are then proven in \S \ref{subsection: Bloch--Kato main results}.

\subsection{Determinants and families of representations}\label{subsection: pseudocharacters and families}
In this subsection, we recall several terminologies for studying families of Galois representations. Most of the materials presented in this subsection are taken from \cite{Bellaiche-Chenevier-book}.

\paragraph{Determinants.} We briefly recall the notion of `\textit{determinants}' from \cite{Chenevier-2014} and refer the readers to \textit{loc. cit.} for more detailed discussions. We remark in the beginning that the notion of determinants is used to strengthen the notion of `\textit{pseudocharacters}' first introduced by R. Taylor in \cite{Taylor-1991} and studied by other mathematicians. We also remark that determinants are equivalent to pseudocharacters in characteristic $0$.

\begin{Definition}\label{Definition: determinants}
Let $A$ be a commutative ring and $R$ be an $A$-algebra (not necessarily commutative). \begin{enumerate}
    \item[(i)] For any $A$-module $M$, one can view $M$ as a functor from the category of commutative $A$-algebras to the category of sets, sending $B$ to $M\otimes_A B$. Let $M$, $N$ be two $A$-modules. Then an \textbf{$A$-polynomial law} between $M$ and $N$ is a natural transformation \[
        M\otimes_A B \rightarrow N\otimes_A B
    \] on the category of commutative $A$-algebras. 
    
    \item[(ii)] Let $P: M \rightarrow N$ be an $A$-polynomial law and $d\in \Z_{>0}$. We say $P$ is \textbf{homogeneous of dimension $d$} if for any commutative $A$-algebra $B$, any $b\in B$ and any $x\in M\otimes_A B$, we have $P(bx) = b^d P(x)$.
    
    \item[(iii)] Let $P: R\rightarrow A$ be an $A$-polynomial law. We say $P$ is \textbf{multiplicative} if, for any commutative $A$-algebra $B$, $P(1) = 1$ and $P(xy) = P(x)P(y)$ for any $x, y\in R\otimes_A B$.
    
    \item[(iv)] For $d\in \Z_{>0}$, a \textbf{$d$-dimensional $A$-valued determinant on $R$} is a multiplicative $A$-polynomial law $D: R\rightarrow A$ which is homogeneous of dimension $d$. 
\end{enumerate}
\end{Definition}

\begin{Example}\label{Example: determinants}
\normalfont Let $G$ be a group and $A$ be any ring. Let $\rho: G \rightarrow \GL_d(A)$ be a representation of dimension $d$. Then \[
    D: A[G] \rightarrow A, \quad G\ni\sigma \mapsto \det\rho(\sigma)
\] is an $A$-valued determinant of dimension $d$ on $A[G]$. We also say that $D$ is an $A$-valued determinant of dimension $d$ on $G$.\qed
\end{Example}

\begin{Theorem}[$\text{\cite[Theorem A \& Theorem B]{Chenevier-2014}}$]\label{Theorem: determinants and representations} Let $G$ be a group.
\begin{enumerate}
    \item[(i)] Let $k$ be an algebraically closed field and let $D: k[G] \rightarrow k$ be a determinant of dimension $d$. Then, there exists a unique (up to isomorphism) semisimple representation $\rho: G \rightarrow \GL_d(k)$ such that for any $\sigma\in G$, we have \[
        \det(1+X\rho(\sigma)) = D(1+X\sigma) \in k[X].
    \] In particular, $\det\rho = D$.
    
    \item[(ii)] Let $A$ be an henselian local ring with algebraically closed residue field $k$, $D: A[G]\rightarrow A$ be a determinant of dimension $d$ and let $\rho$ be the semisimple representation attached to $D\otimes_A k$ in (i). Suppose $\rho$ is irreducible, then there exists a unique (up to isomorphism) representation $\widetilde{\rho}: G \rightarrow \GL_d(A)$ such that \[
        \det(1+X\widetilde{\rho}(\sigma)) = D(1+X\sigma) \in A[X]
    \] for any $\sigma\in G$.
\end{enumerate}
\end{Theorem}

\paragraph{Refinements of crystalline representations.} We recall the notion of `\textit{refinements}' of crystalline representations from \cite[\S 2.4]{Bellaiche-Chenevier-book}. Let $L$ be a finite extension of $\Q_p$ and let $V$ be an $n$-dimensional $L$-representation of $\Gal_{\Q_p}$. Assume that $V$ is crystalline. Also assume that the crystalline Frobenius $\varphi = \varphi_{\cris}$ acting on $\D_{\cris}(V)$ has all eigenvalues living in $L^\times$.

\begin{Definition}[$\text{\cite[\S 2.4.1]{Bellaiche-Chenevier-book}}$]\label{Definition: refinement of crystalline representation}
A \textbf{refinement} of $V$ is the data of a full $\varphi$-stable $L$-filtration \[
    \bbF_{\bullet} : 0 = \bbF_{0} \subsetneq \bbF_1 \subsetneq \cdots \subsetneq \bbF_{n-1}\subsetneq \bbF_{n} = \D_{\cris}(V).
\]
\end{Definition}

Suppose $\bbF_{\bullet}$ is a refinement of $V$, one sees immediately that it determines two orderings: \begin{enumerate}
    \item[(Ref 1)] An ordering $(\varphi_1, ..., \varphi_n)$ of the eigenvalues of $\varphi$ by the formula \[
        \det(X - \varphi|_{\bbF_i}) = \prod_{j=1}^i (X - \varphi_j).
    \] Notice that if the $\varphi_j$'s are all distinct, then such an ordering of eigenvalues of $\varphi$ conversely determines the refinement. 
    
    \item[(Ref 2)] An ordering $(a_1, ..., a_n)$ of Hodge--Tate weights of $V$. More precisely, the jumps of the Hodge filtration of $\D_{\cris}(V)$ induced on $\bbF_i$ are $(a_1, ..., a_i)$.
\end{enumerate}

\begin{Definition}[$\text{\cite[Definition 2.4.5]{Bellaiche-Chenevier-book}}$]\label{Definition: non-critical refiniment}
Suppose the Hodge--Tate weights $a_1 < \cdots < a_n$ of $V$ are all distinct. Let $\bbF$ be a refinement of $V$ and let $\Fil^{\bullet} \D_{\cris}(V)$ be the Hodge filtration of $\D_{\cris}(V)$. We say $\bbF$ is \textbf{non-critical} if, for all $1\leq i \leq n$, we have \[
    \D_{\cris}(V) = \bbF_i \oplus \Fil^{a_i+1} \D_{\cris}(V).
\]
\end{Definition}

Recall the Robba ring \[
    \calR_{L} := \left\{ f(X) = \sum_{i\in \Z} t_n(X-1)^n\in L\llbrack X \rrbrack: \begin{array}{c}
        f(X)\text{ converges on some annulus of $\C_p$ }  \\
        \text{of the form }r(f) \leq |X-1| \leq 1 
    \end{array}\right\}.
\] Here the norm $|\cdot|$ is the $p$-adic norm on $\C_p$ with the normalisation $|p| = 1/p$. Let $\Gamma = \Z_p^\times$. The theory of $(\varphi, \Gamma)$-modules yields an equivalence of categories between the category finite-dimensional $L$-representations of $\Gal_{\Q_p}$ and the category of \'etale $(\varphi, \Gamma)$-modules over $\calR_{L}$ (see, for example, \cite[\S 2.2]{Bellaiche-Chenevier-book}). In particular, we have a $(\varphi, \Gamma)$-module $\D_{\rig}(V)$ over $\calR_{L}$ associated with $V$.

\begin{Proposition}[$\text{\cite[Proposition 2.4.1 \& Proposition 2.4.7]{Bellaiche-Chenevier-book}}$]\label{Proposition: refinements and triangulations}
Let $\bbF_{\bullet}$ be a refinement of $V$. \begin{enumerate}
    \item[(i)] Then $\bbF_{\bullet}$ determines a unique filtration $\Fil_{\bullet} \D_{\rig}(V)$ of length $n$, i.e., a triangulation of $\D_{\rig}(V)$. Consequently, $\bbF_{\bullet}$ determines a unique collection of continuous characters $\delta_i: \Q_p^\times \rightarrow L^\times$ via the isomorphism 
    \[
        \Fil_{i}\D_{\rig}(V)/\Fil_{i-1}\D_{\rig}(V) \simeq \calR_{L}(\delta_i)
    \] given by \cite[Proposition 2.3.1]{Bellaiche-Chenevier-book}. Here, the tuple $\delta = (\delta_1, ..., \delta_n)$ is called the \textbf{parameter} of $V$.
    
    \item[(ii)] Moreover, suppose the Hodge--Tate weight of $V$ are all distinct $h_1 < \cdots < h_n$. Then, $\bbF_{\bullet}$ is non-critical if and only if the sequence Hodge--Tate weights $(a_1, ..., a_n)$ associated with $\bbF_{\bullet}$ in (Ref 2) is increasing, i.e., $a_i = h_i$ for all $i=1, ..., n$.
\end{enumerate}  
\end{Proposition}

\begin{Remark}\label{Remark: trianguline deformation}
\normalfont The theory of $(\varphi, \Gamma)$-modules can be worked out for local artinian $\Q_p$-algebras (see, for example, \cite[\S 2]{Bellaiche-Chenevier-book}). Thus, it makes sense to consider the following deformation functors. Let $\Ar$ be the category of local artinian $\Q_p$-algebras whose residue field is isomorphic to $L$. Then, we define the \textbf{\textit{(local) trianguline deformation}} functor \begin{align*}
    \scrD_{V, \Fil_{\bullet} \D_{\rig}(V)}: \Ar & \rightarrow \Sets, \\
    A & \mapsto \left\{(V_A, \rho_A, \Fil_{\bullet} \D_{\rig}(V_A)): \begin{array}{l}
        V_A \simeq A^n  \\
        \rho_A: \Gal_{\Q_p} \rightarrow \GL(V_A)\simeq \GL_n(A) \text{ s.t. }\rho_A\otimes_A L \simeq V\\
        \Fil_{\bullet} \D_{\rig}(V_A) \otimes_{\calR_{A}} \calR_{L} \simeq \Fil_{\bullet} \D_{\rig}(V)
    \end{array}\right\}/\simeq 
\end{align*} We will also denote the above deformation functor by $\scrD_{V, \bbF_{\bullet}}$ as the triangulation $\Fil_{\bullet} D_{\rig}(V)$ is uniquely determined by $\bbF_{\bullet}$. In fact, we will confuse the refinement $\bbF_{\bullet}$ with the triangulation $\Fil_{\bullet}D_{\rig}(V)$ in what follows. 
\end{Remark}

\paragraph{Families of representations.}
Here, we collect some terminologies introduced in \cite[\S 5]{Bellaiche-Chenevier-book} that will be needed in the later subsections. Note that the terminology of `pseudocharacters' is used in \textit{op. cit.} since the notion of `determinants' was not yet discovered. In what follows, we shall adapt everything with the notion of determinants.

\vspace{2mm}

Let $G$ be a topological group with a continuous group homomorphism $\Gal_{\Q_p}\rightarrow G$, \emph{e.g.}, $G= \Gal_{\Q}$ with the natural inclusion $\Gal_{\Q_p}\hookrightarrow \Gal_{\Q}$. Therefore, any (continuous) representation $\rho$ of $G$ induces a (continuous) representation of $\Gal_{\Q_p}$, denoted by $\rho|_{\Gal_{\Q_p}}$. 

\vspace{1mm}

By a \textbf{\textit{family of representations}}, we mean a datum $(\calX, D)$, where $\calX$ is a reduced separated rigid analytic variety (viewed as an adic space) over $\Spa(\Q_p, \Z_p)$ and a continuous determinant $D: \scrO_{\calX}(\calX)[G] \rightarrow \scrO_{\calX}(\calX)$. The dimension of this family is understood to be the dimension of the determinant $D$, denoted by $n$. For any $\bfitx\in \calX$, let $k_{\bfitx}$ be the residue field of $\bfitx$, then we have the specialisation \begin{equation}\label{eq: specialisation of pseudocharacter}
    D|_{\bfitx} : G \xrightarrow{D} \scrO_{\calX}(\calX) \rightarrow k_{\bfitx}.
\end{equation} Applying Theorem \ref{Theorem: determinants and representations} (i), we see that $D|_{\bfitx}$ is nothing but the determinant of a (unique up to isomorphism) continuous semisimple representation $\rho_{\bfitx}: G \rightarrow \GL_n(\overline{k_{\bfitx}})$. 

\begin{Definition}[$\text{\cite[Definition 4.2.3]{Bellaiche-Chenevier-book}}$]\label{Definition: refined families}
A \textbf{refined family of representations} of dimension $n$ is a datum \((\calX, D, \calQ, \{\alpha_i: i=1,..., n\}, \{F_i: i=1, ..., n\}),\) where \begin{enumerate}
    \item[(a)] $(\calX, D)$ is a family of representations of dimension $n$, 
    \item[(b)] $\calQ \subset \calX$ is a Zariski dense subset, 
    \item[(c)] $\alpha_i\in \scrO_{\calX}(\calX)$ is an analytic function for $i=1, ..., n$, 
    \item[(d)] $F_i\in \scrO_{\calX}(\calX)$ is an analytic function for $i=1, ..., n$,
\end{enumerate} such that \begin{enumerate}
    \item[(i)] For every $\bfitx\in \calX$, the Hodge--Tate--Sen weights\footnote{ Here, the \emph{Hodge--Tate--Sen weight} is defined to be the roots of the Sen polynomial (see, for example, \cite[Definition 2.24]{Liu-trianguline}).} for $\rho_{\bfitx}|_{\Gal_{\Q_p}}$ are $\alpha_1(\bfitx)$, ..., $\alpha_n(\bfitx)$.
    
    \item[(ii)] For each $y\in \calQ$, the representation $\rho_{\bfity}|_{\Gal_{\Q_p}}$ is crystalline (so that $\alpha_i(\bfity)$'s are integers) and we have $\alpha_1(\bfity) < \cdots < \alpha_n(\bfity)$. 
    
    \item[(iii)] For each $\bfity\in \calQ$, the eigenvalues of the crystalline Frobenius $\varphi$ on $\D_{\cris}(\rho_{\bfity}|_{\Gal_{\Q_p}})$ are distinct and are $(p^{\alpha_1(\bfity)} F_1(\bfity), ..., p^{\alpha_n(\bfity)}F_n(\bfity))$.
    
    \item[(iv)] For any $C\in \Z_{> 0}$, define \[
        \calQ_{C} := \left\{\bfity\in \calQ: \begin{array}{l}
            \alpha_{i+1}(\bfity) - \alpha_{i}(\bfity)> C(\alpha_{i}(\bfity) - \alpha_{i-1}(\bfity)) \text{ for }i=2, ..., n-1  \\
            \alpha_2(\bfity) - \alpha_1(\bfity) > C 
        \end{array} \right\}.
    \] We request that $\calQ_C$ accumulates at any point of $\calQ$ for any $C$. In other words, for any $\bfity\in \calQ$ and any $C\in \Z_{>0}$, there is a basis of affinoid neighbourhoods $\calU$ of $\bfitx$ such that $\calU \cap \calQ_C$ is Zariski dense in $\calU$.
    
    \item[(*)] For each $i=1, ..., n$, there is a continuous character $\Z_p^\times \rightarrow \scrO_{\calX}(\calX)^\times$ whose derivative at $1$ is the map $\alpha_i$ and whose evaluation at any point $\bfity\in \calQ$ is the elevation to the $\alpha_i(\bfity)$-th power.
\end{enumerate}
\end{Definition}

Let $(\calX, D, \calQ, \{\alpha_i: i=1, ..., n\}, \{F_i: i=1, ..., n\})$ be a refined family of dimension $n$. We fix a point $\bfity\in \calQ$. Then $\rho_{\bfity}$ admits a natural refinement $\bbF_{\bullet}^{\bfity}$ given by the ordering of distinct eigenvalues \[
    (p^{\alpha_1(\bfity)} F_1(\bfity), ..., p^{\alpha_n(\bfity)} F_n(\bfity))
\] of the crystalline Frobenius acting on $\D_{\cris}(\rho_y|_{\Gal_{\Q_p}})$ (\cite[Definition 4.2.4]{Bellaiche-Chenevier-book}). We assume that $\rho_{\bfity}$ is irreducible and it satisfies the following two conditions:\begin{enumerate}
    \item[(REG)] The refinement $\bbF_{\bullet}^{\bfity}$ is \textbf{\textit{regular}}, i.e., for any $i=1, ..., n$, $p^{\alpha_1(\bfity) + \cdots + \alpha_i(\bfity)}F_1(\bfity) \cdots F_i(\bfity)$ is an eigenvalue of the crystalline Frobenius $\varphi$ acting on $\D_{\cris}(\wedge^i {\rho}_{\bfity}|_{\Gal_{\Q_p}})$ of multiplicity one.
    
    \item[(NCR)] The refinement $\bbF_{\bullet}^{\bfity}$ is non-critical. 
\end{enumerate}
Since $\rho_{\bfity}$ is assumed to be irreducible, Theorem \ref{Theorem: determinants and representations} (ii) implies that there is a unique continuous representation \[
    \rho_{\calX, \bfity}: G \rightarrow \GL_n(\scrO_{\calX, \bfity})
\] such that $\rho_{\calX, \bfity} \otimes_{\scrO_{\calX, \bfity}} k_{\bfity} = \rho_{\bfity}$ and so $\det\rho_{\bfity}$ coincides with the composition $G\xrightarrow{D} \scrO_{\calX}(\calX) \rightarrow \scrO_{\calX, \bfity}$. Following \cite[\S 4.4]{Bellaiche-Chenevier-book}, we define a continuous character $\delta_{\bfity}: \Q_p^\times \rightarrow (\scrO_{\calX, \bfity}^\times)^n$ by setting \begin{equation}\label{eq: parameter ass. with y}
    \delta_{\bfity}(p) = (F_{1, \bfity}, ..., F_{n, \bfity})\quad \text{ and }\quad \delta_{\bfity}|_{\Z_p^\times} = (\alpha_{1, \bfity}^{-1}, ..., \alpha_{n, \bfity}^{-1}),
\end{equation} where $F_{i, \bfity}$ and $\alpha_{i, \bfity}$ are the images of $F_i$ and $\alpha_i$ in $\scrO_{\calX, \bfity}$ respectively.

\begin{Theorem}[$\text{\cite[Theorem 4.4.1]{Bellaiche-Chenevier-book}}$]\label{Theorem: BC refined deformation}
For any ideal $\frakI\subsetneq \scrO_{\calX, \bfity}$ of cofinite length, $\rho_{\calX, \bfity} \otimes_{\scrO_{\calX, \bfity}} \scrO_{\calX, \bfity}/\frakI$ is a trianguline deformation of $(\rho_{\bfity}, \bbF_{\bullet}^{\bfity})$, i.e., it belongs to $\scrD_{\rho_{\bfity}|_{\Gal_{\Q_p}}, \bbF_{\bullet}^{\bfity}}(\scrO_{\calX, \bfity}/\frakI)$ (defined in Remark \ref{Remark: trianguline deformation}), whose parameter is $\delta_{\bfity} \otimes \scrO_{\calX, \bfity}/\frakI$.
\end{Theorem}

\subsection{Galois representations for \texorpdfstring{$\GSp_{2g}$}{GSp}}\label{subsection: Galois representation for GSp}
Given a dominant weight $k = (k_1, ..., k_g)\in \Z_{\geq0}^g$, recall the $\GSp_{2g}$-representations \begin{align*}
    \V_{\GSp_{2g}, k}^{\alg} & = \left\{\phi: \GSp_{2g}(\Q_p)\rightarrow \Q_p: \begin{array}{l}
        \phi\text{ is a polynomial function }  \\
        \phi(\bfgamma\bfbeta) = k(\bfbeta)\phi(\bfgamma) \,\, \forall (\bfgamma, \bfbeta)\in \GSp_{2g}(\Q_p)\times B_{\GSp_{2g}}(\Q_p)  
    \end{array}\right\}\\
    \V_{\GSp_{2g}, k}^{\alg, \vee} & = \Hom_{\Q_p}(\V_{\GSp_{2g}, k}^{\alg}, \Q_p).
\end{align*} The representation $\V_{\GSp_{2g}, k}^{\alg}$ is equipped with a right $\GSp_{2g}(\Q_p)$-action by the formula \[
    \bfgamma \cdot \phi(\bfgamma') = \phi(\bfgamma\bfgamma')
\] for any $\phi\in \V_{\GSp_{2g}, k}^{\alg}$, $\bfgamma, \bfgamma'\in \GSp_{2g}(\Q_p)$. Hence, $\V_{\GSp_{2g}, k}^{\alg, \vee}$ is equipped with a left $\GSp_{2g}(\Q_p)$-action and consequently induces a local system on both $X_{\Iw^+}(\C)$ and $X(\C)$. We abuse the notation and use the same symbol to denote such local system. In particular, we can consider the parabolic cohomology group \[
     H_{\tame,\Par, k}^{\alg, \tol} := \oplus_{t=0}^{2n_0} H_{\Par}^t(X(\C), \V_{\GSp_{2g}, k}^{\alg, \vee}).
\] Note that the double cosets $[\GSp_{2g}(\Z_p) (x\cdot \bfu_{p,i})\GSp_{2g}(\Z_p)]$ acts on $H_{\tame, \Par, k}^{\alg, \tol}$ for any $x\in \Weyl_{\GSp_{2g}}$. We denote by \[
    \bbT^{\tame} := \bbT^p \otimes_{\Z_p} \Z_p\left[[\GSp_{2g}(\Z_p) (x\cdot \bfu_{p,i})\GSp_{2g}(\Z_p)]: i=0, 1, ..., g-1, x\in \Weyl_{\GSp_{2g}}\right].
\] In particular, it makes sense to consider the Hecke polynomial $P_{\Hecke, p}(Y)$ at $p$ in this case and is defined as in (\ref{eq: Hecke poly at l}).

\begin{Hypothesis}\label{Hyp: associated Gal. rep}
For any $\bbT^{\tame}$-eigenclass $[\mu]\in H_{\tame, \Par, k}^{\alg, \tol}$ with eigensystem $\lambda_{[\mu]}: \bbT^{\tame} \rightarrow \overline{\Q}_p$, there exists a (continuous) Galois representation \[
    \rho_{[\mu]}: \Gal_{\Q} \xrightarrow{\rho_{[\mu]}^{\spin}} \GSpin_{2g+1}(\overline{\Q}_p)\xrightarrow{\spin} \GL_{2^g}(\overline{\Q}_p)\footnote{ We refer the readers to \cite[Lecture 20]{Fulton-Harris} for the definition and properties of the representation $\spin: \GSpin_{2g+1} \rightarrow \GL_{2^g}$.}
\] such that \begin{enumerate}
    %\item[(i)] The representation $\rho_{[\mu]}$ is irreducible and there exists a finite extension $F$ of $\Q_p$ such that the image of $\rho_{[\mu]}$ lands in $\GL_{2^g}(F)$. 
    
    \item[(i)] The representation $\rho_{[\mu]}$ is unramified outside $pN$ and \[
     \charpoly(\Frob_{\ell})(Y) = \lambda_{[\mu]}(P_{\Hecke, \ell}(Y)) := \prod_{x\in \Weyl^H} (Y - \lambda_{[\mu]}(T_{\ell, 0}^x))
    \] for any $\ell\nmid pN$, where $\charpoly(\Frob_{\ell})(Y)$ stands for the characteristic polynomial of the Frobenius at $\ell$ and $P_{\Hecke, \ell}(Y)$ is the Hecke polynomial defined in (\ref{eq: Hecke poly at l}). Moreover, the coefficients of these two polynomials are algebraic integers over $\Q$.
    
    \item[(ii)] The representation $\rho_{[\mu]}|_{\Gal_{\Q_p}}$ is crystalline with Hodge--Tate weights \[
    (a_1, ..., a_{2^g}) = (0, a_g', \cdots, a_1', a_g'+a_{g-1}', ..., a_2'+a_1', \cdots, a_g'+\cdots+a_1'), \footnote{ These numbers are all possibilities of sums of $a_i'$'s. The order is chosen so that if $k = (k_1, ..., k_g) = (k_g+g-1, k_g+g-2, ..., k_g)$, we have $a_1 < a_2 < \cdots < a_{2^g}$.}
    \] where $a_i' = (g+1-i)+k_i$. Let $\varphi = \varphi_{\cris}$ be the crystalline Frobenius acting on $\D_{\cris}(\rho_{[\mu]}|_{\Gal_{\Q_p}})$, we moreover have \[
        \charpoly(\varphi)(Y) = \lambda_{[\mu]} (P_{\Hecke, p}(Y)),
    \] where $\charpoly(\varphi)(X)$ is the characteristic polynomial of $\varphi$ acting on $\D_{\cris}(\rho_{[\mu]}|_{\Gal_{\Q_p}})$, and the coefficients of these two polynomials are algebraic integers over $\Q$. We order the eigenvalues of $\varphi$ so they satisfy \[
        (\varphi_1, ..., \varphi_{2^g}) = \varphi_1(1, \varphi_2', ..., \varphi_{g+1}', \varphi_2'\varphi_3', ..., \varphi_{g}'\varphi_{g+1}', ..., \varphi_{2}'\cdots \varphi_{g+1}')
    \] for some $(\varphi_2', ..., \varphi_{g+1}')$. The order of the later tuple is chosen similarly as the Hodge--Tate weights. In particular, $\varphi_2, ..., \varphi_{g+1}$ are divisible by $\varphi_1$ and the $2^g$ eigenvalues of $\varphi$ depend only on $\varphi_1, ..., \varphi_{g+1}$.%The representation $\rho_{[\mu]}|_{G_{\Q_p}}$ is furthermore crystalline if $[\mu]$ is old at $p$.
\end{enumerate}
\end{Hypothesis}

\begin{Remark}\label{Remark: unramified outside bad primes}
\normalfont Recall that $\Sbad$ is the finite set of prime numbers which divides $pN$. Let $\Gal_{\Q, \Sbad}$ be the Galois group of the maximal extension of $\Q$ which is unramified outside $\Sbad$. Therefore, the representation $\rho_{[\mu]}$ in Hypothesis \ref{Hyp: associated Gal. rep} can be regarded as a Galois representation of $\Gal_{\Q, \Sbad}$.
\end{Remark}

\begin{Remark}\label{Remark: on Hyp 1}
\normalfont Evidently, Hypothesis \ref{Hyp: associated Gal. rep} comes from Global Langlands Correspondence. We comment briefly on this hypothesis.  \begin{enumerate}
    \item[(i)] When $g\leq 2$, Hypothesis \ref{Hyp: associated Gal. rep} (i) is well-known (see, for example, \cite{Weissauer}). The work of A. Kret and S. W. Shin (\cite{Kret-Shin}) gave a positive answer to Hypothesis \ref{Hyp: associated Gal. rep} (i) under some conditions on the automorphic representations for general $g$. Although their result is not completely unconditional, it suggests that Hypothesis \ref{Hyp: associated Gal. rep} is reasonable to assume (but could be difficult to prove in general).
    \item[(ii)] Hypothesis \ref{Hyp: associated Gal. rep} (ii) is also well-studied when $g\leq 2$. In particular, E. Urban proved the case for $g=2$ in \cite{Urban-2005}, result deduced from A. Scholl's motive for modular forms (\cite{Scholl-motive}). For general $g$, the property is expected if Hypothesis \ref{Hyp: spin functoriality} below holds (see, for example, \cite[Theorem 2.1 \& Corollary 2.2]{Patrikis--Taylor-2015}).
\end{enumerate}  
\end{Remark}

By \cite[Lemma 0.1]{Kret-Shin} and under the assumption of Hypothesis \ref{Hyp: associated Gal. rep}, we know that given a $\bbT^{\tame}$-eigenclass $[\mu]$ as above, $\rho_{[\mu]}$ factors as \[
    \rho_{[\mu]}: \Gal_{\Q, \Sbad} \xrightarrow{\rho_{[\mu]}^{\spin}} \GSpin_{2g+1}(\overline{\Q}_p) \xrightarrow{\spin}\mathrm{GS}(\overline{\Q}_p) \rightarrow \GL_{2^g}(\overline{\Q}_p),
\] where \[
    \mathrm{GS} =  \left\{\begin{array}{ll}
        \GO_{2^g}, & \text{ if }g(g+1)/2 \text{ is even} \\
        \GSp_{2^g}, & \text{ if }g(g+1)/2 \text{ is odd}
    \end{array}\right.
\] and the last arrow is nothing but the natural inclusion. Define \begin{align*}
    \mathfrak{gl}_{2^g} &:= \text{ the Lie algebra of $\GL_{2^g}(\overline{\Q}_p)$, equipped with the induced adjoint $\Gal_{\Q, \Sbad}$-action by $\rho_{[\mu]}$}\\
    \ad\rho_{[\mu]} &:= \text{ the Lie algebra of $\mathrm{GS}(\overline{\Q}_p)$, equipped with the induced adjoint $\Gal_{\Q, \Sbad}$-action by $\spin\circ \rho_{[\mu]}^{\spin}$}\\
    \ad\rho_{[\mu]}^{\spin} &:= \text{ the Lie algebra of $\GSpin_{2g+1}(\overline{\Q}_p)$, equipped with the induced adjoint $\Gal_{\Q, \Sbad}$-action by $\rho_{[\mu]}^{\spin}$}.
\end{align*} Then, the inclusions \[
    \GSpin_{2g+1}(\overline{\Q}_p) \hookrightarrow \mathrm{GS}(\overline{\Q}_p) \hookrightarrow \GL_{2^g}(\overline{\Q}_p)
\] induces $\Gal_{\Q, \Sbad}$-equivariant inclusions \[
    \ad\rho_{[\mu]}^{\spin} \hookrightarrow \ad\rho_{[\mu]} \hookrightarrow \mathfrak{gl}_{2^g},
\] which then further induces inclusions of the Galois cohomology groups \[
    H^1(\Gal_{\Q, \Sbad}, \ad\rho_{[\mu]}^{\spin}) \hookrightarrow H^1(\Gal_{\Q, \Sbad}, \ad\rho_{[\mu]}) \hookrightarrow H^1(\Gal_{\Q, \Sbad}, \mathfrak{gl}_{2^g}).
\]

\vspace{2mm}

On the other hand, let $\mathfrak{sl}_{2^g}$ be the trace-zero part of $\mathfrak{gl}_{2^g}$ and let \[
    \ad^{0}\rho_{[\mu]} := \ad\rho_{[\mu]} \cap \mathfrak{sl}_{2^g}\quad \text{ and }\quad \ad^0\rho_{[\mu]}^{\spin} := \ad\rho_{[\mu]}^{\spin} \cap \mathfrak{sl}_{2^g}.
\] Note that the decomposition $\mathfrak{gl}_{2^g} = \mathfrak{sl}_{2^g} \oplus \mathfrak{gl}_1$ is $\Gal_{\Q}$-equivariant, we thus have a commutative diagram \begin{equation}\label{eq: inclusions of Galois cohomologies}
    \begin{tikzcd}
        \text{$H^1(\Gal_{\Q, \Sbad}, \ad\rho_{[\mu]}^{\spin})$}\arrow[r, hook] & \text{$H^1(\Gal_{\Q, \Sbad}, \ad\rho_{[\mu]})$} \arrow[r, hook] & H^1(\Gal_{\Q, \Sbad}, \mathfrak{gl}_{2^g})\\
        \text{$H^1(\Gal_{\Q, \Sbad}, \ad^0\rho_{[\mu]}^{\spin})$}\arrow[r, hook]\arrow[u, hook] & \text{$H^1(\Gal_{\Q, \Sbad}, \ad^0\rho_{[\mu]})$} \arrow[r, hook]\arrow[u, hook] & H^1(\Gal_{\Q, \Sbad}, \mathfrak{sl}_{2^g})\arrow[u, hook]
    \end{tikzcd},
\end{equation} where the arrows are all inclusions. 

\vspace{2mm}

Under the assumption of Hypothesis \ref{Hyp: associated Gal. rep}, one obtains a $2^g$-dimensional Galois representation for each eigenclass $[\mu]$. It is then a natural question to ask whether the attached Galois representation admits an associated cuspidal automorphic representation of $\GL_{2^g}$. The answer to this question is expected to be affirmative, which we state as the next hypothesis.

\begin{Hypothesis}[The potential spin functoriality]\label{Hyp: spin functoriality}
Given a $\bbT^{\tame}$-eigenclass $[\mu]\in H_{\tame, \Par, k}^{\alg, \tol}$, there exists a finite real extension $L\subset \overline{\Q}$ of $\Q$ with $\rho_{[\mu]}|_{\Gal_L}$ being irreducible and a generic cuspidal automorphic representation $\pi_{[\mu]}$ of $\GL_{2^g}(\A_{L})$, where $\A_{L}$ is the ring of adèles of $L$, such that\begin{enumerate}
    \item[$\bullet$] $\pi_{[\mu]}$ is unramified outside the places above $\Sbad$ and 
    \item[$\bullet$] the Galois representation associated with $\pi_{[\mu]}$ is isomorphic to $\rho_{[\mu]}|_{\Gal_{L}}$. 
\end{enumerate}
\end{Hypothesis}

\begin{Remark}
\normalfont We should remark that Kret and Shin verify the above hypothesis in \cite[Theorem C]{Kret-Shin} under some stronger conditions than the ones they verify Hypothesis \ref{Hyp: associated Gal. rep}. 
\end{Remark}

On the other hand, we also write \[
    H_{\Par, k}^{\alg, \tol} := \oplus_{t=0}^{2n_0} H_{\Par}^t(X_{\Iw^+}(\C), \V_{\GSp_{2g}, k}^{\alg, \vee}).
\] The forgetful map $X_{\Iw}(\C) \rightarrow X(\C)$ then induces a morphism (see also (\ref{eq: p-old morphism}) and (\ref{eq: p-old morphism for compactly supported cohomology})) \begin{equation}\label{eq: morphism from tame coh. to strict Iwahori coh.}
    \Lambda_p : H_{\tame, \Par, k}^{\alg, \tol} \rightarrow H_{\Par, k}^{\alg, \tol}.
\end{equation} Observe that this morphism is $\bbT^p$-equivariant. Moreover, we have slope decomposition on the latter space with respect to the action of $U_p$ since it is a finite-dimensional $\Q_p$-vector space. Thus, for each $h\in \Q_{>0}$, we write \[
    H_{\tame, \Par, k}^{\alg, \tol, \leq h} := \image\left( H_{\tame, \Par, k}^{\alg, \tol} \xrightarrow{\Lambda_p} H_{\Par, k}^{\alg, \tol} \twoheadrightarrow H_{\Par, k}^{\alg, \tol, \leq h}\right),
\] where $H_{\Par, k}^{\alg, \tol, \leq h}$ is the `$\leq h$' part of $H_{\Par, k}^{\alg, \tol}$ under the action of $U_p$. Thus, for any $\bbT^{\tame}$-eigenclass $[\mu]$ in $H_{\tame, \Par, k}^{\alg, \tol}$, its image in $H_{\tame, \Par, k}^{\alg, \tol, \leq h}$ can be decomposed as a sum of $\bbT$-eigenclasses. We call any of these factors a \textbf{\textit{$p$-stabilisation}} of $[\mu]$.

\vspace{2mm}

It is a natural question asking how the eigenvalues of a $\bbT^{\tame}$-eigenclass interact with the eigenvalues of its $p$-stabilisations. The following statement is due to Harron--Jorza.

\begin{Proposition}[$\text{\cite[Lemma 17]{Harron-Jorza}}$]\label{Prop: Tp and Up eigenvalues}
\begin{enumerate}
    \item[(i)]  Let $[\mu]$ be a $\bbT^{\tame}$-eigenclass with eigensystem $\lambda_{[\mu]}$ in $H_{\Par, k}^{\alg, \tol, \leq h}$. Then, there exist $2^gg!$ $p$-stabilisations $[\mu]^{(p)}$, indexed by $\Weyl_{\GSp_{2g}}$.
    \item[(ii)] Chose a bijection of sets \(
        \iota: \{1, 2, ..., 2^g\} \xrightarrow{\sim} \Weyl^H
    \) so that $\lambda_{[\mu]}\left( T_{p, 0}^{\iota(i)}\right) = \varphi_i$, where $T_{p, 0}^{\iota(i)}$ is the Hecke operator defined by $[\GSp_{2g}(\Z_p) (\iota(i)\cdot \bfu_{p,0}) \GSp_{2g}(\Z_p)]$ acting on $H_{\tame, \Par, \wt(\bfitx)}^{\alg, \tol}$ and $\varphi_i$ is the $i$-th eigenvalue of the crystalline Frobenius associated with $\rho_{[\mu]}$.\footnote{This can be done due to Hypothesis \ref{Hyp: associated Gal. rep} (ii).} Denote by $\lambda_{i} = \lambda_{[\mu]}(T_{p, 0}^{\iota(i)})$ and let $[\mu]^{(p)}$ be any of the $p$-stablisation of $[\mu]$ with Hecke eigensystem $\lambda_{[\mu]}^{(p)}$. Then, there exists a constant $\vartheta \in \Q$ (depending only on $g$) such that, for $i=1, ..., g+1$, \[
    \lambda_{[\mu]}^{(p)} (U_{p,0}^{\iota(i)})= p^{\vartheta}\cdot p^{-(g+1-i)} \lambda_1 \prod_{j=1}^{g}(\lambda_{j+1}/\lambda_1)^{a_{\nu(j)}\text{ or }1-a_{\nu(j)}},
\] where \begin{enumerate}
    \item[$\bullet$] the index of $[\mu]^{(p)}$ is $(\epsilon, \nu)\in \Weyl_{\GSp_{2g}} = \bfSigma_g\ltimes (\Z/2\Z)^{g}$ and \[
        a_{\nu(j)} = \left\{\begin{array}{ll}
            1, &  \nu(j) = i\\
            0, & \text{otherwise}
        \end{array}\right.;
    \]
    \item[$\bullet$] the exponent depends on whether $\epsilon(\nu(j)) = 0 \text{ or }1\in \Z/2\Z$.
\end{enumerate}
\end{enumerate}
\end{Proposition}

\subsection{Families of Galois representations on the cuspidal eigenvariety}\label{subsection: families of representations on eigenvariety}

The goal of this section is to construct families of Galois representations on a sublocus of the cuspidal eigenvariety $\calE_0$ under the assumption of Hypothesis \ref{Hyp: associated Gal. rep}. 

\vspace{2mm}

For any dominant algebraic weight $k\in \Z_{>0}^g$, recall from \cite[Theorem 6.4.1]{Ash-Stevens} that there is $h_k\in \R_{>0}$ such that for any $h\in \Q_{>0}$ with $h<h_k$, we have a canonical isomorphism \[
    H_{\Par, k}^{\tol, \leq h}  \xrightarrow{\sim} H_{\Par, k}^{\alg, \tol, \leq h}.
\] We then define the \textbf{\textit{$p$-stabilised classical locus}} of $\calE_0$ to be the locus $\calX^{\cl}\subset \calE_0$, containing those $\bfitx$ with the following conditions:\begin{enumerate}
    \item[$\bullet$] $\wt(\bfitx) = k\in \Z_{>0}^g$ is a dominant algebraic weight; 
    \item[$\bullet$] there exists $h< h_k$ such that $\bfitx$ corresponds to a $p$-stabilisation of slope $\leq h$ of a $\bbT^{\tame}$-eigenclass $[\mu]$ in $H_{\tame, \Par, k}^{\alg, \tol}$; 
    \item[$\bullet$] the Galois representation $\rho_{[\mu]}^{\spin}$ attached to $[\mu]$ (by Hypothesis \ref{Hyp: associated Gal. rep}) is irreducible. 
\end{enumerate} Consequently, we define \[
    \calE_0^{\irr} := \text{ the Zariski closure of $\calX^{\cl}$ in $\calE_0$.}
\]

\begin{Remark}
\normalfont We do not expect every classical point in $\calE_0$ corresponds to an irreducible Galois representation due to the endoscopy theory of automorphic forms. As we will be only interested in classical points that correspond to irreducible Galois representations, we do not lose information if we only consider $\calE_0^{\irr}$.
\end{Remark}

\begin{Proposition}\label{Proposition: universal trace}
Assume the truthfulness of Hypothesis \ref{Hyp: associated Gal. rep}. \begin{enumerate}
    \item[(i)] For any $\bfitx\in \calX^{\cl}$, there is an associated Galois representation \[
        \rho_{\bfitx}: \Gal_{\Q, \Sbad}\xrightarrow{\rho_{\bfitx}^{\spin}} \GSpin_{2g+1}(\overline{\Q}_p) \xrightarrow{\spin} \GL_{2^g}(\overline{\Q}_p)
    \]
    that satisfies the properties in Hypothesis \ref{Hyp: associated Gal. rep}.
    
    \item[(ii)] There is a universal determinant \[
        \Det^{\univ}: \Gal_{\Q, \Sbad} \rightarrow \scrO_{\calE_0^{\irr}}^+(\calE_0^{\irr})
    \]
    of dimension $2^g$ such that, for any $\bfitx\in \calX^{\cl}$, the specialisation $\Det^{\univ}|_{\bfitx}$ (notation as in (\ref{eq: specialisation of pseudocharacter})) coincides with $\det\rho_{\bfitx}$.
\end{enumerate}
\end{Proposition}
\begin{proof}
The first assertion is easy. Let $\bfitx\in \calX^{\cl}$. It corresponds to a $p$-stabilisation class $[\mu]^{(p)}\in H_{\tame, \Par, k}^{\alg, \tol, \leq h}$. That is, there is a $\bbT^{\tame}$-eigenclass $[\mu]\in H_{\tame, \Par, k}^{\alg, \tol}$ such that $[\mu]^{(p)}$ is a $p$-stabilisation of $[\mu]$. By Hypothesis \ref{Hyp: associated Gal. rep}, the class $[\mu]$ is associated with a Galois representation with desired properties. Then, we define $\rho_{\bfitx}^{\spin} := \rho_{[\mu]}^{\spin}$ and $\rho_{\bfitx} := \rho_{[\mu]}$.\\ %\textcolor{blue}{(well-defined?)}

For the second assertion, we follow the proof of \cite[Proposition 7.1.1]{Chenevier-2004} (see also \cite[Example 2.32]{Chenevier-2014}). Consider the morphism \[
    \Phi: \scrO_{\calE_0^{\irr}}^+(\calE_0^{\irr}) \rightarrow \prod_{\bfitx\in \calX^{\cl}}\C_p, \quad f\mapsto (f(\bfitx))_{\bfitx\in \calX^{\cl}}.
\] Equipped $\prod_{\bfitx\in \calX^{\cl}}\C_p$ with the product topology, one sees that $\Phi$ is continuous. We claim that $\Phi(\scrO^+_{\calE_0^{\irr}}(\calE_0^{\irr}))$ is homeomorphic to $\scrO^+_{\calE_0^{\irr}}(\calE_0^{\irr})$ and is closed in $\prod_{\bfitx\in \calX^{\cl}}\C_p$. Indeed, since $\calX^{\cl}$ is Zariski dense in the reduced space $\calE_0^{\irr}$, the map $\Phi$ is injective. Apply \cite[Corollary 5.4.4]{Johansson-Newton}, we know that $\scrO_{\calE_0^{\irr}}^+(\calE_{0}^{\irr})$ is compact and so $\Phi(\scrO_{\calE_0^{\irr}}^+(\calE_0^{\irr}))$ is closed in $\prod_{\bfitx\in \calX^{\cl}}\C_p$.

\vspace{2mm}

On the other hand, we have a continuous map \[
    \Det : \Gal_{\Q, \Sbad} \rightarrow \prod_{\bfitx\in \calX^{\cl}} \C_p, \quad \sigma \mapsto (\det\rho_{\bfitx}(\sigma))_{\bfitx\in \calX^{\cl}}.
\] One checks easily that $\Det$ is a determinant of dimension $2^g$, in fact, the determinant of a representation $\Gal_{\Q}\rightarrow \GL_{2^g}(\prod_{\bfitx\in \calX^{\cl}}\C_p)$. Hypothesis \ref{Hyp: associated Gal. rep} and $\image \Phi$ being closed in $\prod_{\bfitx\in \calX^{\cl}}\C_p$ imply that $\image\Det \subset \image \Phi$. Hence, we define \[
    \Det^{\univ}:= \Phi^{-1}\circ \Det: \Gal_{\Q, \Sbad} \rightarrow \scrO^+_{\calE_0^{\irr}}(\calE_0^{\irr}).
\] Since $\Phi$ is injective and $\Det$ is a determinant of dimension $2^g$, $\Det^{\univ}$ is as desired.
\end{proof}

\begin{Theorem}\label{Theorem: refined family on the cuspidal eigenvariety}
There exists a subset $\calX_{\heartsuit}^{\cl}\subset \calX^{\cl}$ which is Zariski dense in $\calE_0^{\irr}$, $2^g$ analytic functions $\alpha_1, ..., \alpha_{2^g}\in \scrO_{\calE_0^{\irr}}(\calE_0^{\irr})$ and $2^g$ analytic functions $F_1, ..., F_{2^g}\in \scrO_{\calE_0^{\irr}}(\calE_0^{\irr})$ such that \[
    (\calE_0^{\irr}, \Det^{\univ}, \calX_{\heartsuit}^{\cl}, \{\alpha_i: i=1, ..., 2^g\}, \{F_i: i=1, ..., 2^g\})
\]
is a refined family of Galois representations. 
\end{Theorem}
\begin{proof}
For any $p$-adic weight $\kappa = (\kappa_1, ..., \kappa_g)$, define an ordering of functions on $\Z_p^\times$ via \[
    (\alpha_1, ..., \alpha_{2^g}):=(0 , \alpha_g', ..., \alpha_1', \alpha_g' + \alpha_{g-1}', ..., \alpha_g'+\alpha_1', \alpha_{g-1}'+\alpha_{g-2}', ..., \alpha_2'+\alpha_1', ..., \alpha_g'+\cdots + \alpha_1'), 
\] where $\alpha_i' = (g+1-i)+\kappa_i$ is the character $a\mapsto \kappa_i(a)a^{g+1-i}$ for every $a\in \Z_p^\times$. We can view $\alpha_j$'s as functions on $\calE_0^{\irr}$ by composing with the weight map $\wt$. Obviously from this definition, for any $\bfitx\in \calX^{\cl}$, the functions $\alpha_j$'s provide an ordering of the Hodge--Tate weight of $\rho_{\bfitx}$ in Hypothesis \ref{Hyp: associated Gal. rep} (iii).

\vspace{2mm}

Define\[
    \calX_{\heartsuit}^{\cl} := \left\{\bfitx\in \calX^{\cl}: \begin{array}{l}
        0 = \alpha_1(\bfitx) < \alpha_2(\bfitx) < \cdots < \alpha_{2^g}(\bfitx)  \\
        \text{eigenvalues of the crystalline Frobenius acting on $\D_{\cris}({\rho}_{\bfitx}|_{\Gal_{\Q_p}})$ are distinct} 
    \end{array}\right\}.
\] Observe that $\calX_{\heartsuit}^{\cl}$ is Zariski dense in $\calE_0^{\irr}$ since $\calX^{\cl}$ is Zariski dense in $\calE_0^{\irr}$ and the first condition defining $\calX_{\heartsuit}^{\cl}$ is an open condition on weights while the second condition is an open condition on $\calE_0^{\irr}$. We claim that $\calX_{\heartsuit}^{\cl}$ satisfies condition (iv) in Definition \ref{Definition: refined families}. That is, for any $C\in \Z_{>0}$, we have to show that the set \[
    \calX_{\heartsuit, C}^{\cl} : = \left\{\bfitx\in \calX_{\heartsuit}^{\cl}: \begin{array}{l}
        \alpha_{i+1}(\bfitx) - \alpha_{i}(\bfitx) > C (\alpha_{i}(\bfitx) - \alpha_{i-1}(\bfitx)) \text{ for }i=2, ..., 2^g-1  \\
        \alpha_2(\bfitx) - \alpha_1(\bfitx) >C  
    \end{array}\right\}
\] satisfies that, for any basis of affinoid neighbourhoods $\calV$ of $\bfitx$, $\calV\cap \calX_{\heartsuit, C}^{\cl}$ is Zariski dense in $\calV$. However, this follows from that the condition defining $\calX_{\heartsuit, C}^{\cl}$ is an open condition on the weights. 

\vspace{2mm}

Now, for any $\bfitx\in \calX_{\heartsuit}^{\cl}$, the associated representation $\rho_{\bfitx}$ is crystalline at $p$. Let $\varphi_1(\bfitx), ..., \varphi_{2^g}(\bfitx)$ be eigenvalues of the crystalline Frobenius $\varphi = \varphi_{\cris}$ acting on $\D_{\cris}(\rho_{\bfitx}|_{\Gal_{\Q_p}})$. The order of the eigenvalues $\varphi_i$'s is defined so that it defines a non-critical refinement on $\rho_{\bfitx}$. This is achievable by applying Proposition \ref{Proposition: refinements and triangulations} (ii). Define \[
    F_i(\bfitx) := \varphi_i(\bfitx)/p^{\alpha_i(\bfitx)} \in \C_p.
\] We claim that the collection $\{(F_i(\bfitx))_{i=1, ..., 2^g}\}_{\bfitx\in \calG_{\heartsuit}^{\cl}}$ glue to $2^g$ analytic functions $(F_1, ..., F_{2^g})$ in $\scrO_{\calE_0^{\irr}}(\calE_0^{\irr})$. Let $\lambda_{\bfitx}: \bbT^{\tame} \rightarrow \overline{\Q}_p$ be the eigensystem corresponds to $\bfitx$. Consider \[
    p^{\vartheta}p^{\kappa_{i}'}F_i := \text{ image of the operator $U_{p,0}^{\iota(i)}$ in $\scrO_{\calE_0^{\irr}}(\calE_0^{\irr})$},
\] where \[
    (\kappa_1', ..., \kappa_{2^g}') = (0, \kappa_g, ..., \kappa_1, \kappa_g+\kappa_{g-1}, ..., \kappa_{g}+\kappa_{1}, \kappa_{g-1}+\kappa_{g-2}, ...., \kappa_{2}+\kappa_1, ..., \kappa_g+\cdots+\kappa_1)
\] and $(\kappa_1, ..., \kappa_g) = \wt$ is the weight map. Then, Hypothesis \ref{Hyp: associated Gal. rep} (ii) and Proposition \ref{Prop: Tp and Up eigenvalues} imply the desired result (see also \cite[Proposition 7.5.13]{Bellaiche-Chenevier-book}). 
\end{proof}

\begin{Remark}\label{Remark: eigenvalues of crystalline Frob}
\normalfont Recall that we have ordered the eigenvalues of the crystalline Frobenius $\varphi$ so that they satisfy \[
        (\varphi_1, ..., \varphi_{2^g}) = \varphi_1(1, \varphi_2', ..., \varphi_{g+1}', \varphi_2'\varphi_3', ..., \varphi_{g}'\varphi_{g+1}', ..., \varphi_{2}'\cdots \varphi_{g+1}').
\] On the other hand, recall that $\Weyl^H$ is a set of representatives of $\Weyl_H\backslash \Weyl_{\GSp_{2g}}$, where $\Weyl_{H} \simeq \bfSigma_g$. Observe that $\diag(\one_g, p\one_g)$ is stable under the action of $\bfSigma_g$, thus the action of $\Weyl^H$ on $T_{p, 0}$ only depends on the action of $(\Z/2\Z)^g$. Combining everything together, we have the relation \[
    (F_1, ..., F_{2^g}) = F_1(1, F_2', ..., F_{g+1}', F_{2}'F_{3}', ..., F_{g}'F_{g+1}', ..., F_2'\cdots F_{g+1}').
\] In particular, $F_2, ..., F_{g+1}$ are divisible by $F_1$. 
\end{Remark}

\subsection{Local and global Galois deformations}\label{subsection: Galois deformations}
We keep the notations in the previous subsection. Fix $\bfitx\in \calX_{\heartsuit}^{\cl}$ with $\wt(\bfitx) = k = (k_1, ..., k_g)$ and we write \[
    \rho_{\bfitx}: \Gal_{\Q} \xrightarrow{\rho_{\bfitx}^{\spin}} \GSpin_{2g+1}(\overline{\Q}_p) \xrightarrow{\spin} \GL_{2^g}(\overline{\Q}_p)
\] for the Galois representation attached to $\bfitx$, given by Proposition \ref{Proposition: universal trace}.
We fix a large enough finite field extension $k_{\bfitx}$ of $\Q_p$ such that $k_{\bfitx}$ contains the residue field at $\bfitx$ and  $\rho_{\bfitx}^{\spin}$ takes values in $\GSpin_{2g+1}(k_{\bfitx})$. We also assume that $k_{\bfitx}$ contains all eigenvalues of the Frobenii. 

\vspace{2mm}

Let now $\Ar$ be the category of local artinian $k_{\bfitx}$-algebras whose residue field is $k_{\bfitx}$. We denote by $\bbF_{\bullet}^{\bfitx}$ the refinement of $\rho_{\bfitx}|_{\Gal_{\Q_p}}$ induced by the refined family defined in Theorem \ref{Theorem: refined family on the cuspidal eigenvariety}. We also denote by $\delta = (\delta_1, ..., \delta_{2^g})$ the parameter attached to the triangulation associated with $\bbF_{\bullet}^{\bfitx}$. Notice that the relation of the eigenvalues of crystalline Frobenius and the Hodge--Tate weight implies that the parameter $\delta$ satisfies \[
    (\delta_1, ..., \delta_{2^g}) = \delta_1(1, \delta_2', ..., \delta_{g+1}', \delta_{2}'\delta_{3}', ..., \delta_{g}'\delta_{g+1}', ..., \delta_{2}'\cdots\delta_{g+1}')
\] for some continuous characters $\delta_2', ..., \delta_{g+1}'$ such that $\delta_i = \delta_1\delta_i'$ for all $i=2, ..., g+1$.

\paragraph{Local Galois deformations at $p$.} We shall consider two deformation problems at $p$: \begin{enumerate}
    \item[(i)] The deformation problem \[
         \scrD_{\bfitx, \bbF_{\bullet}^{\bfitx}, p}^{\spin}: \Ar \rightarrow \Sets,
    \] sending each $A\in \Ar$ to the isomorphism classes of representations $\rho_A^{\spin} : \Gal_{\Q_p} \rightarrow \GSpin_{2g+1}(A)$ with a triangulation $\Fil_{\bullet} \D_{\rig}(\spin\circ \rho_A^{\spin})$ such that \begin{enumerate}
        \item[$\bullet$] $\rho_{A}^{\spin} \otimes_A k_{\bfitx} \simeq \rho_{\bfitx}^{\spin}|_{\Gal_{\Q_p}}$;
        \item[$\bullet$] $(\spin \circ \rho_A^{\spin}, \Fil_{\bullet}\D_{\rig}(\spin \circ\rho^{\spin}_A))\in \scrD_{\rho_{\bfitx}|_{\Gal_{\Q_p}}, \bbF_{\bullet}^{\bfitx}}(A)$ and write $\delta_A = (\delta_{A, 1}, ..., \delta_{A, 2^g})$ for the associated parameter;
        \item[$\bullet$] the parameter $\delta_A$ satisfies \[
            (\delta_{A, 1}, ..., \delta_{A, 2^g}) = \delta_{A, 1}( 1, \delta_{A, 2}', ..., \delta_{A, g+1}', \delta_{A, 2}'\delta_{A, 3}', ..., \delta_{A, 2}'\delta_{A, g+1}', \delta_{A, 3}'\delta_{A, 4}', ..., \delta_{A, g}'\delta_{A, g+1}', ..., \delta_{A,2}'\cdots\delta_{A, g+1}')
        \] for some continuous characters $\delta_{A, 2}', ..., \delta_{A, g+1}'$;
        \item[$\bullet$] $\det\spin \circ \rho_A^{\spin} = \det\rho_{\bfitx}|_{\Gal_{\Q_p}}$
    \end{enumerate}
    
    \item[(ii)] The deformation problem \[
        \scrD_{\bfitx, f, p}^{\spin}: \Ar \rightarrow \Sets,
    \] sending each $A\in \Ar$ to the isomorphism classes of representations $\rho_A^{\spin}: \Gal_{\Q_p}\rightarrow \GSpin_{2g+1}(A)$ such that \begin{enumerate}
        \item[$\bullet$] $\rho_A^{\spin} \otimes_A k_{\bfitx} \simeq \rho^{\spin}_{\bfitx}|_{\Gal_{\Q_p}}$;
        \item[$\bullet$] the $(\varphi, \Gamma)$-module $\D_{\rig}(\spin \circ\rho^{\spin}_A)$ is crystalline in the sense of \cite[Definition 2.2.10]{Bellaiche-Chenevier-book} whose eigenvalues $(\varphi_{A, 1}, ..., \varphi_{A, 2^g})$ of the crystalline Frobenius satisfy \begin{align*}
            %& (\alpha_{A, 1}, ..., \alpha_{A, 2^g}) \\
            %& \quad = (0, \alpha_{A, 1}, ..., \alpha_{A, g+1}, \alpha_{A, 2}+\alpha_{A, 3}, ..., \alpha_{A, 2}+\alpha_{A, g+1}, \alpha_{A, 3}+\alpha_{A, 4}, ..., \alpha_{A, g}+\alpha_{A, g+1}, ..., \alpha_{A, 2}+\cdots+\alpha_{A, g+1})\\
            & (\varphi_{A, 1}, ..., \varphi_{A, 2^g}) = \varphi_{A, 1}(1, \varphi_{A, 2}', ..., \varphi_{A, g+1}', \varphi_{A, 2}'\varphi_{A, 3}', ..., \varphi_{A, g}'\varphi_{A, g+1}', ..., \varphi_{A, 2}'\cdots \varphi_{A, g+1}'),
        \end{align*}order chosen the same as for $\varphi_i$'s;
        \item[$\bullet$] $\det \spin\circ \rho_A^{\spin} = \det\rho_{\bfitx}|_{\Gal_{\Q_p}}$
    \end{enumerate}
\end{enumerate}

Consider \[
    L_p' := \ker\left(H^1(\Gal_{\Q_p}, \ad^0\rho^{\spin}_{\bfitx}) \rightarrow H^1(\Gal_{\Q_p}, \ad^0\rho^{\spin}_{\bfitx} \otimes_{k_{\bfitx}} \B_{\cris})\right),
\] where $\B_{\cris}$ is Fontaine's ring of crystalline periods. It is well-known that $L_p'$ defines the tangent space of the crystalline deformation problem for $\rho^{\spin}_{\bfitx}$ with fixed determinant. Consequently, the tangent space $\scrD_{\bfitx, f, p}^{\spin}(k_{\bfitx}[\varepsilon])$, where $\varepsilon$ is a variable such that $\varepsilon^2 = 0$, of $\scrD_{\bfitx, f, p}^{\spin}$ defines a subspace of $L_p'$. Thus, we define \begin{equation}\label{eq: local Selmer factor at p}
    L_p := \scrD_{\bfitx, f, p}^{\spin}(k_{\bfitx}[\varepsilon]) \subset L_p'.
\end{equation}  

\paragraph{Local Galois deformations at $N$.} For any $\ell | N$, we consider the following deformation problem \[
    \scrD_{\bfitx, \ell}^{\spin}: \Ar \rightarrow \Sets
\] sending each $A\in \Ar$ to the isomorphism classes of representations $\rho_A^{\spin}: \Gal_{\Q_{\ell}} \rightarrow \GSpin_{2g+1}(A)$ such that \begin{enumerate}
    \item[$\bullet$] $\rho_A^{\spin} \otimes_{A}k_{\bfitx} \simeq \rho_{\bfitx}^{\spin}|_{\Gal_{\Q_{\ell}}}$;
    \item[$\bullet$] $\rho_{A}^{\spin}|_{I_{\ell}} \simeq \rho_{\bfitx}^{\spin}|_{I_{\ell}} \otimes_{k_{\bfitx}}A$
    \item[$\bullet$] $\det \spin\circ \rho_A^{\spin} = \det\rho_{\bfitx}|_{\Gal_{\Q_{\ell}}}$
\end{enumerate} Here, $I_{\ell}\subset \Gal_{\Q_{\ell}}$ denotes the inertia subgroup. Then, one sees that the tangent space $\scrD_{\bfitx, \ell}(k_{\bfitx}[\varepsilon])$ of $\scrD_{\bfitx, \ell}$ is a $k_{\bfitx}$-subspace of $H^1(\Gal_{\Q_{\ell}}, \ad^0\rho^{\spin}_{\bfitx})$. We consequently define \begin{equation}\label{eq: local factor for Selmer group at N}
    L_{\ell} := \scrD_{\bfitx, \ell}(k_{\bfitx}[\varepsilon]) \subset H^1(\Gal_{\Q_{\ell}}, \ad^0\rho^{\spin}_{\bfitx}).
\end{equation} We learnt the following lemma from P. Allen.

\begin{Lemma}\label{Lemma: Selmer condition at l is trivial}
Under the assumption of Hypothesis \ref{Hyp: spin functoriality}, we have \[
    L_{\ell} = H^1(\Gal_{\Q_{\ell}}, \ad^0\rho_{\bfitx}^{\spin}).
\]
\end{Lemma}
\begin{proof}
Let \[
    H_{\mathrm{unr}}^1(\Gal_{\Q_{\ell}}, \ad^0\rho_{\bfitx}^{\spin}) := \ker\left(H^1(\Gal_{\Q_{\ell}}, \ad^0\rho_{\bfitx}^{\spin}) \rightarrow H^1(I_{\ell}, \ad^0\rho_{\bfitx}^{\spin})\right)
\] By definition, we see that $ H_{\mathrm{unr}}^1(\Gal_{\Q_{\ell}}, \ad^0\rho_{\bfitx}^{\spin}) \subset L_{\ell}$. Thus, it is enough to show that \[
     H_{\mathrm{unr}}^1(\Gal_{\Q_{\ell}}, \ad^0\rho_{\bfitx}^{\spin}) = H^1(\Gal_{\Q_{\ell}}, \ad^0\rho_{\bfitx}^{\spin}).
\]

\vspace{2mm}

First of all, observe that \[
     H_{\mathrm{unr}}^1(\Gal_{\Q_{\ell}}, \ad^0\rho_{\bfitx}^{\spin}) = H^1(\Gal_{\Q_{\ell}}/I_{\ell}, (\ad^0\rho_{\bfitx}^{\spin})^{I_{\ell}})
\] by definition. Note that $\Gal_{\Q_{\ell}}/I_{\ell} \simeq \widehat{\Z}$. Hence, one deduces from the discussion in \cite[Chapter XIII, \S 1]{Serre-LocalFields} that \begin{align*}
    \dim_{k_{\bfitx}} H_{\mathrm{unr}}^1(\Gal_{\Q_{\ell}}, \ad^0\rho_{\bfitx}^{\spin}) & = \dim_{k_{\bfitx}} H^1(\Gal_{\Q_{\ell}}/I_{\ell}, (\ad^0\rho_{\bfitx}^{\spin})^{I_{\ell}})\\
    & = \dim_{k_{\bfitx}} H^0(\Gal_{\Q_{\ell}}/I_{\ell}, (\ad^0\rho_{\bfitx}^{\spin})^{I_{\ell}})\\
    & = \dim_{k_{\bfitx}} H^0(\Gal_{\Q_{\ell}}, \ad^0\rho_{\bfitx}^{\spin}).
\end{align*} By applying the local Euler characteristic, the desired equation will follow once we show \[
    H^2(\Gal_{\Q_{\ell}}, \ad^0\rho_{\bfitx}^{\spin}) = 0.
\] By Tate duality, it is equivalent to show \[
    H^0(\Gal_{\Q_{\ell}}, \ad^0\rho_{\bfitx}^{\spin}(1)) = 0.
\]

\vspace{2mm}

Let $L$ be the real extension of $\Q$ as in Hypothesis \ref{Hyp: spin functoriality}, we claim that for any place $v$ in $L$ sitting above $\ell$, we have \[
    H^0(\Gal_{L_{v}}, \ad^0\rho_{\bfitx}^{\spin}(1)) = 0,
\] where $\Gal_{L_v} = \Gal(\overline{\Q}_{\ell}/L_v)$ is the absolute Galois group of $L_v$. However, under the assumption of Hypothesis \ref{Hyp: spin functoriality}, the desired vanishing follows from \cite[Lemma 1.3.2]{BLGGT14} and the discussion around (\ref{eq: inclusions of Galois cohomologies}). 

\vspace{2mm}

Finally, observe that the restriction map \[
    \Res: H^0(\Gal_{\Q_{\ell}}, \ad^0\rho_{\bfitx}^{\spin}(1)) \rightarrow H^0(\Gal_{L_{v}}, \ad^0\rho_{\bfitx}^{\spin}(1))
\] is an injection since $k_{\bfitx}$ is of characteristic zero so that \[
    \mathrm{Corres}\circ \Res = \text{ multiplication by }[L_v : \Q_{\ell}]
\] is an injection. The assertion then follows. 
\end{proof}

\paragraph{Global Galois deformations.} Consider the following two global deformation functors:\begin{enumerate}
    \item[(i)] The deformation problem \[
        \scrD_{\bfitx, \bbF_{\bullet}^{\bfitx}}^{\spin} : \Ar \rightarrow \Sets,
    \] sending each $A\in \Ar$ to isomorphism classes of representations $\rho^{\spin}_A: \Gal_{\Q, \Sbad} \rightarrow \GSpin_{2g+1}(A)$ and triangulation $\Fil_{\bullet}\D_{\rig}(\spin \circ \rho^{\spin}_A|_{\Gal_{\Q_p}})$ such that \begin{enumerate}
        \item[$\bullet$] $\rho^{\spin}_A\otimes_A k_{\bfitx} \simeq \rho^{\spin}_{\bfitx}$
        \item[$\bullet$] $\det \spin \circ \rho_A^{\spin} = \det\rho_{\bfitx}$
        \item[$\bullet$] $(\spin \circ \rho^{\spin}_{A}|_{\Gal_{\Q_p}}, \Fil_{\bullet}\D_{\rig}(\spin \circ \rho_A|_{\Gal_{\Q_p}}))\in \scrD^{\spin}_{\bfitx, \bbF_{\bullet}^{\bfitx}, p}(A)$
        \item[$\bullet$] $\rho^{\spin}_A|_{\Gal_{\Q_{{\ell}}}}\in \scrD^{\spin}_{\bfitx, \ell}(A)$ for $\ell\in \Sbad$
    \end{enumerate}
    
    \item[(ii)] The deformation problem \[
        \scrD^{\spin}_{\bfitx, f}: \Ar \rightarrow \Sets,
    \] sending each $A\in \Ar$ to isomorphism classes of representations $\rho^{\spin}_A: \Gal_{\Q, \Sbad} \rightarrow \GSpin_{2g+1}(A)$ such that \begin{enumerate}
        \item[$\bullet$] $\rho^{\spin}_A \otimes_{k_{\bfitx}} k_{\bfitx} \simeq \rho^{\spin}_{\bfitx}$
        \item[$\bullet$] $\det \spin \circ \rho_A = \det\rho_{\bfitx}$
        \item[$\bullet$] $\rho^{\spin}_A|_{\Gal_{\Q_p}}\in \scrD^{\spin}_{\bfitx, f, p}(A)$
        \item[$\bullet$] $\rho^{\spin}_A|_{\Gal_{\Q_{\ell}}}\in \scrD^{\spin}_{\bfitx, \ell}(A)$ for $\ell\in \Sbad$.
    \end{enumerate}
\end{enumerate} 

\begin{Lemma}\label{Lemma: representibility of the deformation functors} Keep the above notations. 
\begin{enumerate}
    \item[(i)] The deformation problems $\scrD^{\spin}_{\bfitx, \bbF_{\bullet}^{\bfitx}}$ and $\scrD^{\spin}_{\bfitx, f}$ are pro-representable. Denote by $R_{\bfitx, \bbF_{\bullet}^{\bfitx}}^{\univ}$ and $R_{\bfitx, f}^{\univ}$ the complete noetherian local rings that represent these two deformation functors respectively.  
    \item[(ii)] Suppose $\bbF_{\bullet}^{\bfitx}$ is non-critical, then $\scrD^{\spin}_{\bfitx, f}$ is a subfunctor of $\scrD^{\spin}_{\bfitx, \bbF_{\bullet}^{\bfitx}}$.
\end{enumerate}
\end{Lemma}
\begin{proof}
Since $\rho_{\bfitx}$ is absolutely irreducible, the first assertion follows from standard Galois deformation theory (see, for example, \cite[\S 4]{Khare-Thorne-2017} and \cite[Proposition 3.7 \& Proposition 3.8]{Hansen-Thorne}). The second assertion is an immediate consequence of \cite[Proposition 2.5.8]{Bellaiche-Chenevier-book}. Notice that our deformation problems are slightly different from the ones considered in \textit{op. cit.} and \cite{Hansen-Thorne}. In fact, one sees easily that our deformation problems are subfunctors of the deformation problems considered therein. Their results imply ours since $\spin: \GSpin_{2g+1} \rightarrow \GL_{2^g}$ is a closed immersion, the conditions we required on the relations of the parameters and the fixed determinant of the deformations are closed conditions and they are stable under isomorphisms, \emph{i.e.}, they satisfy the definition of `deformation problems' (see, for example, \cite[Definition 4.1]{Khare-Thorne-2017}).
\end{proof}

The Bloch--Kato Selmer group associated with $\ad^0\rho^{\spin}_{\bfitx}$ is defined to be \begin{equation}\label{eq: Bloch--Kato Selmer group}
    H_f^1(\Q, \ad^0\rho^{\spin}_{\bfitx}) := \ker\left(H^1(\Gal_{\Q, \Sbad}, \ad^0\rho^{\spin}_{\bfitx}) \xrightarrow{\Res} \prod_{\ell\in \Sbad\cup \{p\}} \frac{H^1(\Gal_{\Q_{\ell}}, \ad^0\rho^{\spin}_{\bfitx})}{L_{\ell}}\right),
\end{equation} where $L_{\ell}$ are as defined in (\ref{eq: local Selmer factor at p}) and (\ref{eq: local factor for Selmer group at N}).

\begin{Proposition}\label{Proposition: deformation and Bloch-Kato}
The tangent space $\scrD_{\bfitx, f}^{\spin}(k_{\bfitx}[\varepsilon])$ of $\scrD^{\spin}_{\bfitx, f}$ can naturally be identified with the Bloch--Kato Selmer group $H_f^1(\Q, \ad^0\rho^{\spin}_{\bfitx})$.
\end{Proposition}
\begin{proof}
This is follows from standard Galois deformation theory (see, for example, \cite[Proposition 3.7]{Hansen-Thorne}) and the definition of $L_p$ and $L_{\ell}$ (see \eqref{eq: local Selmer factor at p} and \eqref{eq: local factor for Selmer group at N}). 
\end{proof}

\subsection{The adjoint Bloch--Kato Selmer groups}\label{subsection: Bloch--Kato main results}
We keep the notations and assumptions in the previous subsection. We further assume the following \begin{enumerate}
    \item[$\bullet$] the refinement $\bbF_{\bullet}^{\bfitx}$ of $\rho_{\bfitx}$ satisfies (REG) and (NCR);\footnote{ In fact, the condition (NCR) is already satisfied by the definition of $\calX_{\heartsuit}^{\cl}$.}
    \item[$\bullet$] the representation $\rho_{\bfitx}|_{\Gal_{\Q_p}}$ is not isomorphic to its twist by the $p$-adic cyclotomic character. 
\end{enumerate}

\begin{Lemma}\label{Lemma: Galois representation at p for the Hecke algebra}
Denote by $\bbT_{\bfitx} := \widehat{\scrO}_{\calE_0^{}, \bfitx}$ the completed local ring at $\bfitx$. Then, for any ideal of cofinite length $\frakI\subset \bbT_{\bfitx}$ there exists a Galois representation \[
    \rho_{\frakI}: \Gal_{\Q, \Sbad} \rightarrow \GL_{2^g}(\bbT_{\bfitx}/\frakI)
\]
such that \begin{enumerate}
    \item[(i)] $\rho_{\frakI} \otimes_{\bbT_{\bfitx}} k_{\bfitx} \simeq \rho_{\bfitx}$
    \item[(ii)] $\rho_{\frakI}|_{\Gal_{\Q_p}} \in \scrD_{\rho_{\bfitx}|_{\Gal_{\Q_p}}, \bbF_{\bullet}^{\bfitx}, p}(\bbT_{\bfitx}/\frakI)$
\end{enumerate}
\end{Lemma}
\begin{proof}
The first assertion is a consequence of Theorem \ref{Theorem: determinants and representations}. The second assertion is a consequence of Theorem \ref{Theorem: BC refined deformation}.
\end{proof}

\begin{Hypothesis}\label{Hypothesis: minimal ramification} Consider the Galois representation $\rho_{\frakI}$ in Lemma \ref{Lemma: Galois representation at p for the Hecke algebra} for any ideal of cofinite length $\frakI\subset \bbT_{\bfitx}$. We assume
\begin{enumerate}
    \item[(i)] The Galois representation $\rho_{\frakI}$ factors as \[
        \rho_{\frakI}: \Gal_{\Q, \Sbad} \xrightarrow{\rho_{\frakI}^{\spin}} \GSpin_{2g+1}(\bbT_{\bfitx}/\frakI) \xrightarrow{\spin} \GL_{2^g}(\bbT_{\bfitx}/\frakI).
    \]
    \item[(ii)] The Galois representation $\rho^{\spin}_{\frakI}|_{\Gal_{\Q_p}}\in \scrD^{\spin}_{\bfitx, \bbF_{\bullet}^{\bfitx}, p}(\bbT_{\bfitx}/\frakI)$.
    
    \item[(iii)] The tame level structure $\Gamma^{(p)}$ implies that the Galois representation $\rho^{\spin}_{\frakI}$ satisfies \[
    \rho^{\spin}_{\frakI}|_{\Gal_{\Q_{\ell}}} \in \scrD^{\spin}_{\bfitx, \ell}(\bbT_{\bfitx}/\frakI)
\] for any $\ell | N$.
\end{enumerate}
\end{Hypothesis}

\begin{Remark}\label{Remark: hypothesis on minimal ramification}
\normalfont We remark that the above hypothesis is safe to assume:\begin{enumerate}
    \item[(i)] The first two conditions are natural. When $g=1$, the conditions are trivial. When $g=2$, $\GSpin_{5}$ is isomorphic to $\GSp_{4}$. In this case, the proof of \cite[Lemma 4.3.3]{Genestier-Tilouine} implies the conditions.
    \item[(ii)] Roughly speaking, the third condition in the hypothesis means that the level structure on the automorphic side determines the ramification type on the Galois side. This condition is inspired by the Taylor--Wiles method. When $g=1$, the classical example is the work of R. Taylor and A. Wiles in \cite{Taylor--Wiles}. In \emph{loc. cit.}, they showed that if one considers the Hecke algebra on the space of weight-2 modular forms of a certain level, then the Galois representation with coefficients in the local Hecke algebra satisfies certain Galois deformation problem. For higher-rank groups, one sees, for example, such a relation in \cite[\S 4.3]{Genestier-Tilouine} for $\GSp_4$ and \cite[\S 3.4]{Clozel-Harris-Taylor} for $\GL_n$ over CM fields.
\end{enumerate}  
\end{Remark}

\begin{Lemma}\label{Lemma: specialisation of the deformation ring}
Denote by $R_{\wt(\bfitx)}$ the complete local ring at $\wt(\bfitx)$ and so we have a natural homomorphism $R_{\wt(\bfitx)} \rightarrow \Q_p\rightarrow k_{\bfitx}$, where the first map is given by quotienting the maximal ideal and the second map is the natural inclusion. Then, $R_{\bfitx, \bbF_{\bullet}^{\bfitx}}^{\univ}$ admits an action of $R_{\wt(\bfitx)}$ and \[
    R_{\bfitx, \bbF_{\bullet}^{\bfitx}}^{\univ}\otimes_{R_{\wt(\bfitx)}} k_{\bfitx} = R_{\bfitx, f}^{\univ}.
\]
\end{Lemma}
\begin{proof}
Let us first explain the action of $R_{\wt(\bfitx)}$ on $R_{\bfitx, \bbF_{\bullet}^{\bfitx}}^{\univ}$. For any $A\in \Ar$, observe that we have a natural morphism \[
    \scrD^{\spin}_{\bfitx, \bbF_{\bullet}^{\bfitx}}(A) \rightarrow \Hom_{\cts}(T_{\GL_g,1}, A^\times), \quad \rho_A^{\spin}\mapsto ((\delta_{A,g+1}')^{-1}|_{\Z_p^\times}-g, (\delta_{A, g}')^{-1}|_{\Z_p^{\times}} - (g-1), ..., (\delta_{A,2}')^{-1}|_{\Z_p^\times}-1).
\] Under this map, the image of $\rho_{\bfitx}^{\spin}$ is exactly $k = (k_1, ..., k_g)$ by (\ref{eq: parameter ass. with y}). Consequently, there is a natural morphism \[
    \Z_p\llbrack T_{\GL_g, 1}\rrbrack \rightarrow R_{\bfitx, \bbF_{\bullet}^{\bfitx}}^{\univ},
\] which factors through $R_{\wt(\bfitx)}$.

\vspace{2mm}

Since the refinement $\bbF_{\bullet}^{\bfitx}$ satisfies (REG), together with the relation of parameters and the condition of fixed determinant, the desired isomorphism follows from the constant weight lemma (\cite[Proposition 2.5.4]{Bellaiche-Chenevier-book}), \emph{i.e.}, the crystalline deformations of $\rho_{\bfitx}$ are of constant Hodge--Tate weight, of which being the same as $\rho_{\bfitx}$.
\end{proof}

\begin{Lemma}\label{Lemma: exact sequence for tangent spaces}
Denote by $H_{\bbF_{\bullet}^{\bfitx}}^1(\Q, \ad^0\rho^{\spin}_{\bfitx})$ the tangent space $\scrD^{\spin}_{\bfitx, \bbF_{\bullet}^{\bfitx}}(k_{\bfitx}[\varepsilon])$ of $\scrD^{\spin}_{\bfitx, \bbF_{\bullet}^{\bfitx}}$. We have an exact sequence \[
    0 \rightarrow H_f^1(\Q, \ad^0\rho^{\spin}_{\bfitx}) \rightarrow H_{\bbF_{\bullet}^{\bfitx}}^1(\Q, \ad^0\rho^{\spin}_{\bfitx}) \rightarrow k_{\bfitx}^g.
\]
\end{Lemma}
\begin{proof}
Following \cite[Proposition 7.6.4]{Bellaiche-Chenevier-book}, we expect an exact sequence \[
    0 \rightarrow H_f^1(\Q, \ad^0\rho^{\spin}_{\bfitx}) \rightarrow H_{\bbF_{\bullet}^{\bfitx}}^1(\Q, \ad^0\rho^{\spin}_{\bfitx}) \rightarrow k_{\bfitx}^{2^g}.
\] The first map is clear while the second map is defined as follows. For any $A\in \Ar$, we have \[
    \scrD^{\spin}_{\bfitx, \bbF_{\bullet}^{\bfitx}}(A) \rightarrow \Hom_{\cts}(\Q_p^\times, A^\times)^{2^g}, \quad \rho_A\mapsto (\delta_{A, 1}, ..., \delta_{A, 2^g}).
\]  Composing with the derivative at $1$, we obtain a morphism \[
    \scrD^{\spin}_{\bfitx, \bbF_{\bullet}^{\bfitx}}(A) \rightarrow A^{2^g}.
\] That is, we obtain \[
    \partial: \scrD^{\spin}_{\bfitx, \bbF_{\bullet}^{\bfitx}} \rightarrow \widehat{\bbG_m^{2^g}}.
\] The second map is then defined to be $\partial(k_{\bfitx}[\varepsilon])$. Lemma \ref{Lemma: specialisation of the deformation ring} shows that $H_{f}^1(\Q, \ad^0\rho^{\spin}_{\bfitx}) = \ker\partial(k_{\bfitx}[\varepsilon])$.

\vspace{2mm}

Recall that the local condition of $\scrD^{\spin}_{\bfitx, \bbF_{\bullet}^{\bfitx}}$ at $p$ requires a relation of the parameters and a fixed determinant. Thus, the image of $\partial(k_{\bfitx}[\varepsilon])$ lies in a subspace of dimension $g$, depending only on the continuous characters $\delta_{A, 2}'$, ..., $\delta_{A, g+1}'$.
\end{proof}

\begin{Proposition}\label{Proposition: relation between R and T}
Retain the notation in Lemma \ref{Lemma: Galois representation at p for the Hecke algebra} and assume Hypothesis \ref{Hypothesis: minimal ramification} holds. 
\begin{enumerate}
    \item[(i)] There exists a canonical ring homomorphism $R^{\univ}_{\bfitx, \bbF_{\bullet}^{\bfitx}}\rightarrow \bbT_{\bfitx}$.
    
    \item[(ii)] If the adjoint Bloch--Kato Selmer group $H_f^1(\Q, \ad^0\rho_{\bfitx}^{\spin})$ vanishes, then the canonical map in (i) is an isomorphism $R^{\univ}_{\bfitx, \bbF_{\bullet}^{\bfitx}}\simeq \bbT_{\bfitx}$ (an `infinitesimal $R=T$ theorem').
\end{enumerate}
\end{Proposition}
\begin{proof}
By Lemma \ref{Lemma: Galois representation at p for the Hecke algebra} and Hypothesis \ref{Hypothesis: minimal ramification}, for any ideal $\frakI\subset \bbT_{\bfitx}$ of cofinite length, there is a canonical ring homomorphism \[
    R_{\bfitx, \bbF_{\bullet}^{\bfitx}}^{\univ} \rightarrow \bbT_{\bfitx}/\frakI.
\] This ring homomorphism is surjective due to the fact that the characteristic polynomials of the Frobenii under $\rho_{\frakI}$ are given by the Hecke polynomials. Consequently, one obtains a canonical morphism \[
    R_{\bfitx, \bbF_{\bullet}^{\bfitx}}^{\univ} \rightarrow \bbT_{\bfitx} = \varprojlim_{\frakI\text{ : cofinite length}} \bbT_{\bfitx}/\frakI
\] with dense image. Since $R_{\bfitx, \bbF_{\bullet}^{\bfitx}}^{\univ}$ is complete, the canonical morphism $R_{\bfitx, \bbF_{\bullet}^{\bfitx}}^{\univ} \rightarrow \bbT_{\bfitx}$ is surjective. 

\vspace{2mm}

Finally, if $H_f^1(\Q, \ad^0\rho_{\bfitx}^{\spin})$ vanishes, then the exact sequence in Lemma \ref{Lemma: exact sequence for tangent spaces} implies that \[
    \dim_{k_{\bfitx}} H_{\bbF_{\bullet}^{\bfitx}}^1(\Q, \ad^0\rho^{\spin}_{\bfitx}) \leq g.
\] Since $R_{\bfitx, \bbF_{\bullet}^{\bfitx}}^{\univ}$ is a local noetherian ring, its Krull dimension is bounded by the dimension of its tangent space (\cite[\href{https://stacks.math.columbia.edu/tag/00KD}{Section 00KD}]{stacks-project}), \emph{i.e.}, $\dim R_{\bfitx, \bbF_{\bullet}^{\bfitx}}^{\univ} \leq g$. Moreover, we also know from \textit{loc. cit.} that the equality holds if and only if $R_{\bfitx, \bbF_{\bullet}^{\bfitx}}^{\univ}$ is regular. However, since $\calE_0$ is equidimensional and finite over $\calW$, we know that $\dim\bbT_{\bfitx} = \dim \calW = g$. Therefore, \[
    g \geq \dim R_{\bfitx, \bbF_{\bullet}^{\bfitx}}^{\univ} \geq \dim \bbT_{\bfitx} = g
\] and $R_{\bfitx, \bbF_{\bullet}^{\bfitx}}^{\univ}$ is regular of dimension $g$. To conclude the proof, suppose $\fraka = \ker(R_{\bfitx, \bbF_{\bullet}^{\bfitx}}^{\univ} \rightarrow \bbT_{\bfitx})$ is non-zero and so we can identify $\bbT_{\bfitx}$ with $R_{\bfitx, \bbF_{\bullet}^{\bfitx}}^{\univ}/\fraka$. Since $R_{\bfitx, \bbF_{\bullet}^{\bfitx}}^{\univ}$ is a regular local ring, it is a domain (\cite[\href{https://stacks.math.columbia.edu/tag/00NP}{Lemma 00NP}]{stacks-project}). We then obtain a contradiction\[
    g = \dim R_{\bfitx, \bbF_{\bullet}^{\bfitx}}^{\univ} > \dim R_{\bfitx, \bbF_{\bullet}^{\bfitx}}^{\univ}/\fraka = \dim \bbT_{\bfitx} = g.
\]
\end{proof}

Due to the nice property stated in Proposition \ref{Proposition: relation between R and T}, we will from now on assume the truthfulness of Hypothesis \ref{Hypothesis: minimal ramification}.

\begin{Corollary}\label{Corollary: main result for BK Selmer group}
Suppose Hypothesis \ref{Hyp: associated Gal. rep}, Hypothesis \ref{Hyp: spin functoriality}, and Hypothesis \ref{Hypothesis: minimal ramification} hold. Assume the following also hold:\begin{enumerate}
    \item[$\bullet$] The cuspidal automorphic representation $\pi_{\bfitx}$ of $\GL_{2^g}(\A_L)$ associated with $\rho_{\bfitx}$ as in Hypothesis \ref{Hyp: spin functoriality} is regular algebraic and polarised (see, for example, \cite[\S 2.1]{BLGGT14}). 
    \item[$\bullet$] The image $\rho_{\bfitx}(\Gal_{L(\zeta_{p^{\infty}})})$ is enormous (see \cite[Definition 2.27]{Newton--Thorne}).
\end{enumerate} Then \begin{enumerate}
    \item[(i)]  $H_f^1(\Q, \ad^0\rho_{\bfitx}^{\spin}) = 0$; and
    \item[(ii)] $R_{\bfitx, \bbF_{\bullet}^{\bfitx}}^{\univ} \simeq \bbT_{\bfitx}$.
\end{enumerate}  
\end{Corollary}
\begin{proof}
By the discussion around (\ref{eq: inclusions of Galois cohomologies}), we have \[
    H_f^1(\Q, \ad^0\rho_{\bfitx}^{\spin}) \subset H_f^1(\Q, \ad^0\rho_{\bfitx}).
\] However, the latter space vanishes by \cite[Theorem 5.3]{Newton--Thorne} and so we conclude by Proposition \ref{Proposition: relation between R and T}.
\end{proof}

We conclude this paper with another situation when one can also deduce the vanishing of the adjoint Bloch--Kato Selmer group. 

\begin{Corollary}\label{Corollary: etale + R=T implies vanishing of adjoint BK}
Suppose Hypothesis \ref{Hyp: associated Gal. rep} and Hypothesis \ref{Hypothesis: minimal ramification} hold. Suppose the weight map is étale at $\bfitx$ and suppose the canonical morphism $R_{\bfitx, \bbF_{\bullet}^{\bfitx}}^{\univ} \rightarrow \bbT_{\bfitx}$ is an isomorphism. Then, \[
    H_f^1(\Q, \ad^0\rho_{\bfitx}^{\spin}) = 0.
\]
\end{Corollary}
\begin{proof}
Observe the following sequence of isomorphisms \begin{align*}
    \Omega_{\bbT_{\bfitx}/R_{\wt(\bfitx)}}^1\otimes_{\bbT_{\bfitx}} k_{\bfitx} & \simeq \Omega_{R_{\bfitx, \bbF_{\bullet}^{\bfitx}}^{\univ}/R_{\wt(\bfitx)}}^1\otimes_{R_{\bfitx, \bbF_{\bullet}^{\bfitx}}^{\univ}} k_{\bfitx}\\
    & \simeq \Omega_{R_{\bfitx, \bbF_{\bullet}^{\bfitx}}^{\univ}/R_{\wt(\bfitx)}}^1 \widehat{\otimes}_{R_{\bfitx, \bbF_{\bullet}^{\bfitx}}^{\univ}} R_{\bfitx, f}^{\univ} \otimes_{R_{\bfitx, f}^{\univ}} k_{\bfitx}\\
    & \simeq \Omega_{R_{\bfitx, \bbF_{\bullet}^{\bfitx}}^{\univ}/R_{\wt(\bfitx)}}^1 \widehat{\otimes}_{R_{\bfitx, \bbF_{\bullet}^{\bfitx}}^{\univ}} R_{\bfitx, \bbF_{\bullet}^{\bfitx}}^{\univ} \otimes_{R_{\wt(\bfitx)}} k_{\bfitx} \otimes_{R_{\bfitx, f}^{\univ}} k_{\bfitx}\\
    & \simeq \Omega_{R_{\bfitx, \bbF_{\bullet}^{\bfitx}}^{\univ}/R_{\wt(\bfitx)}}^1 \otimes_{R_{\wt(\bfitx)}} k_{\bfitx} \otimes_{R_{\bfitx, f}^{\univ}} k_{\bfitx}\\
    & \simeq \Omega_{R_{\bfitx, f}^{\univ}/k_{\bfitx}}^1 \otimes_{R_{\bfitx, f}^{\univ}} k_{\bfitx}.
\end{align*} Here, the first isomorphism follows from the assumption $R_{\bfitx, \bbF_{\bullet}^{\bfitx}}^{\univ} \simeq \bbT_{\bfitx}$ and the third and the final isomorphism follows from Lemma \ref{Lemma: specialisation of the deformation ring}. Therefore, we have \begin{align*}
    \dim_{k_{\bfitx}} H_f^1(\Q, \ad^0\rho_{\bfitx}^{\spin}) & = \dim_{k_{\bfitx}} \Hom_{k_{\bfitx}}(\Omega_{R_{\bfitx, f}^{\univ}/k_{\bfitx}}^1 \otimes_{R_{\bfitx, f}^{\univ}} k_{\bfitx}, k_{\bfitx})\\
    & = \dim_{k_{\bfitx}} \Hom_{k_{\bfitx}}(\Omega_{\bbT_{\bfitx}/R_{\wt(\bfitx)}}^1\otimes_{\bbT_{\bfitx}} k_{\bfitx}, k_{\bfitx})\\
    & = \dim_{k_{\bfitx}} \Omega_{\bbT_{\bfitx}/R_{\wt(\bfitx)}}^1\otimes_{\bbT_{\bfitx}} k_{\bfitx}\\
    & \leq \length_{\bbT_{\bfitx}}\Omega_{\bbT_{\bfitx}/R_{\wt(\bfitx)}}^1.
\end{align*} However, since the weight map is étale at $\bfitx$, $\length_{\bbT_{\bfitx}}\Omega_{\bbT_{\bfitx}/R_{\wt(\bfitx)}}^1 = 0$. We then conclude the result.
\end{proof}

\begin{Remark}\label{Remark: relation to the pairing}
\normalfont 
Suppose we are now working with the strict Iwahori level Siegel modular variety and suppose the \emph{$p$-adic adjoint $L$-function} $L^{\adj}$ in \cite{Wu-pairing} is defined at $\bfitx$. Suppose we are also in the situation of Corollary \ref{Corollary: etale + R=T implies vanishing of adjoint BK}. Then, by \cite[Theorem 4.3.5]{Wu-pairing}, we then have \[
    \ord_{\bfitx} L^{\adj} = 0 = \dim_{k_{\bfitx}} H_f^1(\Q, \ad^0\rho_{\bfitx}^{\spin}).
\] Such a relation then (conjecturally) justifies the name of $L^{\adj}$. More generally, in light of the Bloch--Kato conjecture (Conjecture \ref{Conjecture: BK}), we expect that, if $\bfitx$ is a smooth point, \[
    \ord_{\bfitx} L^{\adj} = \dim_{k_{\bfitx}} H_f^1(\Q, \ad^0 \rho_{\bfitx}^{\spin}).
\] In particular, since $H_f^1(\Q, \ad^0 \rho_{\bfitx}^{\spin})$ is expected to vanish, it seems fair to expect that, if $\bfitx$ is a smooth point with small slope and at which $L^{\adj}$ is defined, the weight map is étale at $\bfitx$. When $g=1$, this is \cite[Theorem 2.16]{Bellaiche-critical}.  
\end{Remark}

\subsection{Appendix: Examples}\label{subsection: examples}
In our main result (Corollary \ref{Corollary: main result for BK Selmer group}), many assumptions are made. The purpose of this appendix is to provide examples as evidence that we did not make vacuum assumptions. For the convenience of the readers, we recall the suppositions: 
\begin{enumerate}
    \item[$\bullet$]  The point $\bfitx\in \calE_0$ corresponds to a $p$-stabilisation of an eigenclass of tame level. 
    \item[$\bullet$] Hypothesis \ref{Hyp: associated Gal. rep} holds. In particular, one can attach a $\GSpin_{2g+1}$-valued Galois representation $\rho_{\bfitx}^{\spin}$ to $\bfitx$, which is crystalline when restricting to $\Gal_{\Q_p}$. And we let $\rho_{\bfitx} := \spin \circ \rho_{\bfitx}^{\spin}$.
    \item[$\bullet$] Hypothesis \ref{Hyp: spin functoriality} holds, \emph{i.e.}, the potential spin functoriality holds. Moreover, the cuspidal automorphic representation $\pi_{\bfitx}$ of $\GL_{2^g}(\A_L)$ is regular algebraic and polarised.
    \item[$\bullet$] Hypothesis \ref{Hypothesis: minimal ramification} holds, \emph{i.e.}, the Galois representation valued in the Hecke algebra satisfies the desired deformation conditions. 
    \item[$\bullet$] The restriction $\rho_{\bfitx}|_{\Gal_{\Q_p}}$ admits a refinement $\bbF_{\bullet}^{\bfitx}$ that satisfies (REG) and (NCR). 
    \item[$\bullet$] The restriction $\rho_{\bfitx}|_{\Gal_{\Q_p}}$ is not isomorphic to its twist by the $p$-adic cyclotomic character. 
    \item[$\bullet$] The image $\rho_{\bfitx}(\Gal_{L(\zeta_{p^{\infty}})})$ is enormous. 
\end{enumerate}

In what follows, we discuss examples for $g=1, 2$. In these cases, Hypothesis \ref{Hyp: associated Gal. rep} is well-understood by mathematicians (see Remark \ref{Remark: on Hyp 1}) and so we will skip the discussions. Note also that Hypothesis \ref{Hyp: spin functoriality} is trivial when $g=1$.

\begin{Example}
\normalfont 
Our first example concerns $g=1$ and suppose $\bfitx$ corresponds to a $p$-stabilisation of a weight-$k$ normalised newform $f = \sum_{n>0}a_nq^n$ of level $\Gamma(N)$ with $p\nmid N$ and $p > k \geq 2$. We assume that $f$ is not a CM form. The Hecke polynomial of $f$ at $p$ is given by \[
    Y^2 - a_pY + p^{k-1}.
\] We assume that the two roots $\alpha$, $\beta$ are distinct. 

In this case, by the result in \cite{FaltingsHT}, we know that the associated Galois representation \[
    \rho_{\bfitx} = \rho_f: \Gal_{\Q} \rightarrow \GL_2(\overline{\Q}_p)
\]
of $f$ is irreducible and of Hodge--Tate weight $(0, k-1)$ at $p$. Moreover, $\rho_f|_{\Gal_{\Q_p}}$ is crystalline by \cite[Theorem 1.2.4]{Scholl-motive}.

Let's now check the conditions imposed on $\rho_f$. First of all, it is easy to see that $\rho_f|_{\Gal_{\Q_p}}$ is not isomorphic to its twist by the $p$-adic cyclotomic character. Moreover, since $\rho_f|_{\Gal_{\Q_p}}$ is a 2-dimensional crystalline representation, it satisfies (NCR) by \cite[Remark 2.4.6]{Bellaiche-Chenevier-book}. To check (REG), note that the characteristic polynomial of the crystalline Frobenius $\varphi$ is equal to the Hecke polynomial at $p$ (\cite[Theorem 1.2.4]{Scholl-motive}). Since $\alpha \neq \beta$, $\alpha$ and $\alpha\beta = p^{k-1}$ are eigenvalues of the crystalline Frobenii on $\D_{\cris}(\rho_{f}|_{\Gal_{\Q_p}})$ and $\D_{\cris}(\wedge^2 \rho_{f}|_{\Gal_{\Q_p}})$ respectively with multiplicity one.  Finally, combining the result in \cite{Newton--Thorne_symmetric} and \cite[Example 2.3.4]{Newton--Thorne}, we know that $\rho_f(\Gal_{\Q(\zeta_{p^{\infty}})})$ is enormous. 

It remains to check Hypothesis \ref{Hypothesis: minimal ramification}. The first point in Hypothesis \ref{Hypothesis: minimal ramification} is trivial. Additionally, the second point follows from that the Galois representations of finite-slope overconvergent eigenforms are triangulline (\cite[Theorem 6.3]{Kisin-2003} and \cite[Proposition 4.3]{Colmez-triangulline}). Finally, by the discussions in \cite[\S 3.2]{DFG}, we know that the deformation valued in the Hecke algebra is minimally ramified at $\ell|N$ and hence the last point in Hypothesis \ref{Hypothesis: minimal ramification}. 
\hfill\qedsymbol
\end{Example}

\begin{Example}
\normalfont
In this example, we let $g=2$ and suppose $\bfitx$ corresponds to a $p$-stabilisation of a discrete series cuspidal automorphic representation $\pi_{\GSp_4}$ of $\GSp_{4}(\A_{\Q})$ which is spherical at $p$ and of cohomological weight $k = (k_1, k_2)$ with $k_1 \geq k_2 \geq 0$. We assume $\pi_{\GSp_4}$ is neither CAP nor endoscopic. 

In this case, the Galois representation \[
    \rho_{\bfitx}^{\spin} = \rho_{\pi_{\GSp_4}}: \Gal_{\Q} \rightarrow \GSpin_5(k_{\bfitx}) \simeq \GSp_4(k_{\bfitx})
\]
associated to $\pi_{\GSp_4}$ is irreducible, where $k_{\bfitx}$ is a large enough finite extension of $\Q_p$. The Hodge--Tate weight of $\rho_{\pi_{\GSp_4}}|_{\Gal_{\Q_p}}$ is $(0, k_2+1, k_1+2, k_1+k_2+3)$. Moreover, $\rho_{\pi_{\GSp_4}}|_{\Gal_{\Q_p}}$ is crystalline and its characteristic polynomial of the crystalline Frobenius coincides with the Hecke polynomial at $p$. We impose the following assumptions on the Galois representation $\rho_{\pi_{\GSp_4}}$: 
\begin{enumerate}
    \item[$\bullet$] We assume that the Hecke polynomial at $p$ is decomposed as \[
        P_{\Hecke, p} = (Y-\alpha)(Y-\beta)(Y-\gamma)(Y-\delta)
    \] with distinct roots $\alpha$, $\beta$, $\gamma$, $\delta$. Note that all the roots are of Weil weight $k_1+k_2-3$ (\cite[Theorem 1]{Weissauer}).
    \item[$\bullet$] Let $\Sbad$ be the finite set, consisting of primes at which $\pi_{\GSp_4}$ is not spherical.\footnote{ Note the different definitions of $\Sbad$ here and the main body of the paper.} We assume that if $\ell\in \Sbad$, then the restriction of the residual representation $\overline{\rho}_{\pi_{\GSp_4}}|_{I_{\ell}}$ is absolutely irreducible and $p\nmid \ell^{12}-1$.
\end{enumerate}

Note that the spin representation $\spin: \GSpin_5 \rightarrow \GL_4$ is nothing but the natural embedding of $\GSp_4 \hookrightarrow \GL_4$. Following the discussion in \cite[\S 2]{DZ-bigimage}, we know that there is a cuspidal automorphic representation $\pi_{\GL_4}$ of $\GL_4(\A_{\Q})$, which is regular algebraic and polarised,\footnote{ The terminology `polarised' is called `essentially self-dual' in \emph{op. cit.}.} such that the associated Galois representation $\pi_{\GL_4}$ is  \[
    \rho_{\bfitx} = \rho_{\pi_{\GL_4}} : \Gal_{\Q} \xrightarrow{\rho_{\pi_{\GSp_4}}} \GSp_4(\overline{\Q}_p) \hookrightarrow \GL_4(\overline{\Q}_p).
\]

Let's now check the conditions on $\rho_{\pi_{\GL_4}}$. First of all, one sees that $\rho_{\pi_{\GL_4}}|_{\Gal_{\Q_p}}$ is not isomorphic to its twist by the $p$-adic cyclotomic character by comparing the Hodge--Tate weights on both sides. Next, since $0 < k_2+1 < k_1+2 < k_1+k_2+3$, we can apply \cite[Proposition 2.4.7]{Bellaiche-Chenevier-book} and know that (NCR) is satisfied. Moreover, since $\alpha$, $\beta$, $\gamma$, $\delta$ are distinct but with the same Weil weight, we see that (REG) is also satisfied.

We show that $\rho_{\pi_{\GL_4}}(\Gal_{\Q(\zeta_{p^{\infty}})})$ is enormous. First, note that if the Zariski closure of $\rho_{\pi_{\GL_4}}(\Gal_{\Q(\zeta_p^{\infty})})$ in $\GL_4$ contains $\Sp_4$, then it is enormous by \cite[Lemma 2.33]{Newton--Thorne}. Using the strategy in [\emph{op. cit}, Example 2.34], it is enough to show that the Zariski closure of $\rho_{\pi_{\GL_4}}(\Gal_{\Q})$ in $\GL_4$ contains $\Sp_4$. However, since $\pi_{\GSp_4}$ is neither CAP nor endoscopic, the desired result follows from the discussion in \cite[\S 9.3.4]{Hida-Tilouine_2015}.

Finally, we check Hypothesis \ref{Hypothesis: minimal ramification}. The first point holds by the argument of \cite[Lemma 4.3.3]{Genestier-Tilouine}. The second point holds due to the fact that the Galois representations of finite-slope overconvergent Siegel modular forms of genus 2 are triangulline (\cite[Theorem 13.3]{Conti-2019}). For the third point, we first remark that, by \cite[Lemma 4.3.6]{Genestier-Tilouine}, $\rho_{\pi_{\GSp_4}}(I_{\ell}) = \rho_{\pi_{\GL_4}}(I_{\ell})$ is finite of order prime to $p$. Hence, we can verify Hypothesis \ref{Hypothesis: minimal ramification} (iii) by the following lemma: 

\begin{Lemma}
Let $A\in \Ar$ and let $\rho_A: \Gal_{\Q} \rightarrow \GSp_4(A)$ be a representation such that $\rho_A \otimes_{A}k_{\bfitx} \simeq \rho_{\pi_{\GSp_4}}$. Then, for $\ell\in \Sbad$, $\rho_A(I_{\ell})$ is finite of order prime to $p$. In particular, we have \[
    \rho_A(I_{\ell}) \simeq \rho_{\pi_{\GSp_4}}(I_{\ell}).
\]
\end{Lemma}
\begin{proof}
The proof of this lemma is basically \cite[Lemma 4.3.3]{Genestier-Tilouine}. 

Note first that $\ker(\GSp_4(A) \rightarrow \GSp_4(k_{\bfitx}))$ is a locally pro-$p$ group. Hence, it suffices to show that $\rho_A(I_{\ell})$ is finite of order prime to $p$. 

To show this, we make a further reduction. Let $I_{\ell}^{(\ell)}$ be the pro-$\ell$ Sylow subgroup of $I_{\ell}$. That is, $I_{\ell}^{(\ell)}$ is the Galois group of the maximal tamely ramified extension $\Q_{\ell}^{\mathrm{tame}}$ of $\Q_{\ell}$. Since $I_{\ell}^{(\ell)}$ is pro-$\ell$ and $\GSp_4(A)$ is locally pro-$p$, the image $\rho_A(I_{\ell}^{(\ell)})$ is finite. Let $\widetilde{F}$ be the finite Galois extension of $\Q_{\ell}^{\mathrm{tame}}$ defined by $\ker\rho_A|_{I_{\ell}^{(\ell)}}$. Then, there exists a finite Galois extension $F$ of $\Q_{\ell}^{\mathrm{unr}}$ such that $\widetilde{F} = F\Q_{\ell}^{\mathrm{tame}}$. Moreover, if we let $I_{F, \ell} := \Gal(\overline{\Q}_{\ell}/F)$ and $I_{F, \ell}^{(\ell)} := \ker \rho_A|_{I_{\ell}^{(\ell)}}$, then \[
    I_{\ell}/I_{\ell}^{(\ell)} \simeq I_{F, \ell}/I_{F, \ell}^{(\ell)}
\]
and so it suffice to show know that $\rho_A(I_{F, \ell})$ is finite, providing $\rho_A|_{I_{F, \ell}^{(\ell)}}$ being trivial.

Recall that \[
     I_{F, \ell}/I_{F, \ell}^{(\ell)} \simeq I_{\ell}/I_{\ell}^{(\ell)} \simeq \prod_{q\neq \ell}\Z_{q}(1).
\] Therefore, via the isomorphisms above, we only need to show $\rho_A(\Z_p(1))$ is trivial. We prove this in the following two steps:

Let $\xi \in \Z_p(1)$ be a topological generator. We first claim that $\rho_A(\xi)$ is unipotent. Suppose $\rho_A(\xi)$ is not unipotent, then it would admit an eigenvalue $\epsilon \neq 1$. By conjugating with the $\rho_A(\Frob_{\ell})$, we see that $\rho_A(\xi)$ and $\rho_A(\xi^{\ell})$ have same eigenvalues. By iterating such a process, we learn that $\{\epsilon, \epsilon^{\ell}, \epsilon^{\ell^2}, \epsilon^{\ell^3}\}$ is a subset of eigenvalues of $\rho_A(\xi)$ while $\{\epsilon^{\ell}, \epsilon^{\ell^2}, \epsilon^{\ell^3}, \epsilon^{\ell^4}\}$ is a subset of eigenvalues of $\rho_A(\xi^{\ell})$. Comparing these two sets, one deduces the identity \[
    \epsilon = \epsilon^{\ell^{12}}.
\]
In particular, $\epsilon$ is a root of unity. On the other hand, since $\xi$ is a topological generator of the pro-$p$ group, $\epsilon$ can only be a $p$-power root of unity. Thus, we have \[
    p| \ell^{12} -1,
\] 
which contradicts to the assumption that $p\nmid \ell^{12} - 1$.

Finally, we claim that $\rho_A(\xi) = 1$. If $\rho_A(\xi) \neq 1$, then it would fix a subspace $V$ of $A^4$, which is stable under the action of $I_{\ell}$. On the other hand, since $\overline{\rho}_{\pi_{\GSp_4}}|_{I_{\ell}}$ is irreducible, $\rho_{\pi_{\GSp_{4}}}|_{I_{\ell}}$ is irreducible and so is $\rho_A|_{I_{\ell}}$. The existence of $V$ then contradicts the irreducibility. 
\end{proof}
\hfill\qed
\end{Example}

\printbibliography[heading=bibintoc]

@book{Serre-LocalFields,
    author = {Jean-Pierre Serre}, 
    title = {Local Fields}, 
    series = {Graduate Texts in Mathematics}, 
    volume = {67}, 
    year = {1979}, 
    publisher = {Springer-Verlag New York}, 
    doi = {10.1007/978-1-4757-5673-9}
}

@article{Bellaiche-critical,
    author = {Joël Bella\"{i}che},
    title = {Critical $p$-adic $L$-functions}, 
    journal = {Inventiones mathematicae}, year = {2012},
    volume = {189},
    pages = {1-–60},
    doi = {https://doi.org/10.1007/s00222-011-0358-z},
}

@book{Bellaiche-Chenevier-book,
    author = {Bella\"iche, Jo\"el and Chenevier, Ga\"etan},
    title = {Families of Galois representations and Selmer groups},
    series = {Ast\'erisque},
    publisher = {Soci\'et\'e math\'ematique de France},
    number = {324},
    year = {2009},
    zbl = {1192.11035},
    mrnumber = {2656025},
    %language = {en},
    %url = {www.numdam.org/item/AST_2009__324__R1_0/}
}

@misc{stacks-project,
  author       = {The {Stacks project authors}},
  %shorthand    = {Stacks},
  title        = {The Stacks project},
  howpublished = {\url{https://stacks.math.columbia.edu}},
  year         = {2022},
}

@article{Hansen-PhD, 
    author = {David Hansen},
    title = {Universal eigenvarieties, trianguline Galois representations, and $p$-adic Langlands functoriality},
    journal = {Journal für die reine und angewandte Mathematik}, 
    year = {2017},
    doi = {https://doi.org/10.1515/crelle-2014-0130},
    issue = {730},
    pages = {1--64}
}

@article{Johansson-Newton,
    title = "Extended eigenvarieties for overconvergent cohomology",
    author = "Christian Johansson and James Newton",
    year = "2019",
    month = "2",
    day = "13",
    doi = "10.2140/ant.2019.13.93",
    %language = "English",
    volume = "13",
    pages = "93--158",
    journal = "Algebra and Number Theory",
    issn = "1937-0652",
    publisher = "Mathematical Sciences Publishers",
    number = "1",
}

@book{Fulton-Harris,
    author = {William Fulton and Joe Harris}, 
    title = {Representation Theory: a first course}, 
    publisher = {Springer, New York}, 
    series = {Graduate Texts in Mathematics},
    volume = {129}, 
    year = {1991},
    doi = {https://doi.org/10.1007/978-1-4612-0979-9}
}

@book{Faltings-Chai,
    author = {Faltings, Gred and Chai, Ching-Li}, 
    title = {Degeneration of Abelian Varieties},
    series = {Ergebnisse der Mathematik und ihrer Grenzgebiete},
    DOI = {10.1007/978-3-662-02632-8},
    publisher = {Springer-Verlag Berlin Heidelberg}, 
    year = {1990}
}

@misc{Ash-Stevens,
    author = {Avner Ash and Glenn Stevens},
    title = {$p$-adic deformations of arithmetic cohomology},
    howpublished = {Preprint. Available at: \url{http://math.bu.edu/people/ghs/preprints/Ash-Stevens-02-08.pdf}},
    year = {2008}
}

@article{Kisin-2003,
    author = {Mark Kisin}, 
    title = {Overconvergent modular forms and the Fontaine--Mazur conjecture}, 
    journal = {Inventiones mathematicae},
    volume = {153}, 
    year = {2003}, 
    page = {373–-454},
}

@article {FaltingsHT,
    AUTHOR = {Faltings, Gerd},
     TITLE = {Hodge-{T}ate structures and modular forms},
   JOURNAL = {Math. Ann.},
  FJOURNAL = {Mathematische Annalen},
    VOLUME = {278},
      YEAR = {1987},
    NUMBER = {1-4},
     PAGES = {133--149},
      ISSN = {0025-5831},
   MRCLASS = {11F85 (11G25 14G20)},
  MRNUMBER = {909221},
       DOI = {10.1007/BF01458064},
       URL = {https://doi-org.lib-ezproxy.concordia.ca/10.1007/BF01458064},
}

@incollection{Genestier-Tilouine,
    author = {Genestier, Alain and Tilouine, Jacques},
    title = {Syst\`emes de Taylor-Wiles pour $\GSp_4$},
    booktitle = {Formes automorphes (II) - Le cas du groupe $\GSp(4)$},
    editor = {Jacques Tilouine and Henri Carayol and Michael Harris and Marie-France Vign\'eras},
    series = {Ast\'erisque},
    publisher = {Soci\'et\'e math\'ematique de France},
    number = {302},
    year = {2005},
    pages = {177-290},
    %zbl = {1142.11036},
    %mrnumber = {2234862},
    %language = {fr},
    %url = {http://www.numdam.org/item/AST_2005__302__177_0}
}

@misc{Kret-Shin,
    title = {Galois representations for the general symplectic group},
    author = {Arno Kret and  Sug Woo Shin}, howpublished = {To appear in \textit{Journal of the European Mathematical Society}. Preprint available at: \url{https://arxiv.org/abs/1609.04223}},
    year = {2020}
}

@article{Scholl-motive,
    author = {Scholl, Anthony J.},
    journal = {Inventiones mathematicae},
    %keywords = {-adic parabolic cohomology groups; Kuga-Sato varieties; Grothendieck motive; Hecke polynomial},
    number = {2},
    pages = {419--430},
    title = {Motives for modular forms.},
    %url = {http://eudml.org/doc/143792},
    volume = {100},
    year = {1990},
}

@incollection{Urban-2005,
    author = {Urban, Eric},
    title = {Sur les repr\'esentations $p$-adiques associ\'ees aux repr\'esentations cuspidales de $\GSp_{4/\mathbb{Q}}$},
    booktitle = {Formes automorphes (II) - Le cas du groupe $\GSp(4)$},
    editor = {Jacques Tilouine and Henri Carayol and Michael Harris and Marie-France Vign\'eras},
    series = {Ast\'erisque},
    publisher = {Soci\'et\'e math\'ematique de France},
    number = {302},
    year = {2005},
    pages = {151-176},
    %zbl = {1100.11017},
    %mrnumber = {2234861},
    %language = {fr},
    %url = {http://www.numdam.org/item/AST_2005__302__151_0}
}

@incollection{Weissauer,
     author = {Weissauer, Rainer},
     title = {Four dimensional Galois representations},
     booktitle = {Formes automorphes (II) - Le cas du groupe $\GSp(4)$},
     editor = {Jacques Tilouine and Henri Carayol and Michael Harris and Marie-France Vign\'eras},
     series = {Ast\'erisque},
     publisher = {Soci\'et\'e math\'ematique de France},
     number = {302},
     year = {2005},
     zbl = {1097.11027},
     mrnumber = {2234860},
     %language = {en},
     %url = {www.numdam.org/item/AST_2005__302__67_0/}
}

@article{Harron-Jorza, 
    author = {Robert Harron and Andrei Jorza}, 
    title = {On symmetric power $\calL$-invariants of Iwahori level Hilbert modular forms}, 
    journal = {American Journal of Mathematics}, 
    year = {2017}, 
    volume = {139}, 
    number = {6},
    pages = {1605--1647},
    doi = {doi:10.1353/ajm.2017.0040}
}

@article{Clozel-Harris-Taylor,
    author = {Clozel, Laurent and Harris, Michael and Taylor, Richard},
    title = {Automorphy for some $\ell$-adic lifts of automorphic mod $\ell$ Galois representations},
    journal = {Publications Math\'ematiques de l'IH\'ES},
    publisher = {Springer-Verlag},
    volume = {108},
    year = {2008},
    pages = {1--181},
    doi = {10.1007/s10240-008-0016-1},
    %zbl = {1169.11020},
    %mrnumber = {2470687},
    %language = {en},
    %url = {http://www.numdam.org/item/PMIHES_2008__108__1_0}
}

@article{Khare-Thorne-2017, 
    author = {Chandrashekhar B. Khare and Jack A. Thorne}, 
    tile = {Potential Automorphy and the Leopoldt conjecture}, 
    year = {2017}, 
    journal = {American Journal of Mathematics}, 
    volume = {139}, 
    number = {5},
    pages = {1205-1273}, 
    doi = {10.1353/ajm.2017.0030}
}

@article{Hansen-Thorne,
    author = {David Hansen and Jack Thorne},
    title = {On the $\GL_n$-eigenvariety and a conjecture of Venkatesh},
    journal = {Selecta Mathematica}, 
    volume = {23}, 
    pages = {1205--1234}, 
    year = {2017},
    doi ={https://doi.org/10.1007/s00029-017-0303-0},
}

@article{Newton--Thorne, 
    title = {Adjoint Selmer groups of automorphic Galois representations of unitary type}, 
    author = {James Newton and Jack Thorne}, 
    journal = {Journal of the European Mathematical Society}, 
    doi = {https://doi.org/10.17863/CAM.55603}, 
    year= {2020}
}

@article{Newton--Thorne_symmetric, 
    title = {Symmetric power functoriality for holomorphic modular forms}, 
    author = {James Newton and Jack Thorne}, 
    journal = {Publications mathématiques de l'IHÉS}, 
    doi = {https://doi.org/10.1007/s10240-021-00127-3}, 
    year= {2021}, 
    volume = {134}, 
    pages = {1 -- 116}
}

@article{BLGGT14, 
    %shorthand = {BLGGT14},
    author = {Thomas Barnet-Lamb and Toby Gee and David Geraghty and Richard Taylor}, 
    title = {Potential automorphy and change of weight}, 
    doi = {https://doi.org/10.4007/annals.2014.179.2.3}, 
    journal = {Annals of Mathematics}, 
    year = {2014},
}

@article{Taylor-1991,
    author = "Taylor, Richard",
    doi = "10.1215/S0012-7094-91-06312-X",
    %fjournal = "Duke Mathematical Journal",
    journal = "Duke Math. J.",
    %month = "07",
    number = "2",
    pages = "281--332",
    publisher = "Duke University Press",
    title = "Galois representations associated to Siegel modular forms of low weight",
    %url = "https://doi.org/10.1215/S0012-7094-91-06312-X",
    volume = "63",
    year = "1991"
}

@article{Taylor--Wiles,
    title = {Ring-Theoretic Properties of Certain Hecke Algebras}, 
    author = {Richard Taylor and Andrew Wiles}, 
    journal = {Annals of Mathematics}, 
    volume = {141}, 
    number = {3}, 
    year = {1995}
}

@article {Chenevier-2004,
    author = "Gaëtan Chenevier",
    title = "Familles $p$-adiques de formes automorphes pour $\GL_n$",
    journal = "Journal für die reine und angewandte Mathematik",
    year = "2004",
    publisher = "De Gruyter",
    address = "Berlin, Boston",
    volume = "2004",
    number = "570",
    doi = "https://doi.org/10.1515/crll.2004.031",
    pages=      "143 - 217",
    %url = "https://www.degruyter.com/view/journals/crll/2004/570/article-p143.xml"
}

@inbook{Chenevier-2014,
    %place={Cambridge}, 
    series={London Mathematical Society Lecture Note Series}, 
    title={The $p$-adic analytic space of pseudocharacters of a profinite group and pseudorepresentations over arbitrary rings}, 
    volume={1}, 
    DOI={10.1017/CBO9781107446335.008}, 
    booktitle={Automorphic Forms and Galois Representations}, 
    publisher={Cambridge University Press}, 
    author={Chenevier, Gaëtan}, 
    editor={Diamond, Fred and Kassaei, Payman L. and Kim, Minhyong}, 
    year={2014}, 
    pages={221–285}, 
    collection={London Mathematical Society Lecture Note Series}
}

@article{Liu-trianguline,
    author = {Ruochuan Liu}, 
    title = {Triangulation of refined families}, 
    journal = {Commentarii Mathematici Helvetici}, 
    volume = {90}, 
    number = {4}, 
    doi = {10.4171/CMH/372}, 
    year = {2015}, 
    pages = {831–904},
}

@article{Patrikis--Taylor-2015,
    title={Automorphy and irreducibility of some l-adic representations}, 
    volume={151}, DOI={10.1112/S0010437X14007519}, 
    number={2}, 
    journal={Compositio Mathematica}, 
    publisher={London Mathematical Society}, 
    author={Patrikis, Stefan and Taylor, Richard}, 
    year={2015}, 
    pages={207 -- 229}
}

@article{DFG,
    author = {Diamond, Fred and Flach, Matthias and Guo, Li},
    title = {The {Tamagawa} number conjecture of adjoint motives of modular forms},
    journal = {Annales scientifiques de l'\'Ecole Normale Sup\'erieure},
    pages = {663--727},
    publisher = {Elsevier},
    volume = {Ser. 4, 37},
    number = {5},
    year = {2004},
    doi = {10.1016/j.ansens.2004.09.001},
    %zbl = {02136287},
    %mrnumber = {2103471},
    %language = {en},
    %url = {http://www.numdam.org/articles/10.1016/j.ansens.2004.09.001/}
}

@incollection{Colmez-triangulline,
    author = {Colmez, Pierre},
    title = {Repr\'esentations triangulines de dimension $2$},
    booktitle = {Repr\'esentations $p$-adiques de groupes $p$-adiques I : repr\'esentations galoisiennes et $(\varphi, \Gamma)$-modules},
    editor = {Berger Laurent and Breuil Christophe and Colmez Pierre},
    series = {Ast\'erisque},
    publisher = {Soci\'et\'e math\'ematique de France},
    number = {319},
    year = {2008},
}

@article{Conti-2019, 
    title={Galois level and congruence ideal for $p$-adic families of finite slope Siegel modular forms}, 
    volume={155}, 
    DOI={10.1112/S0010437X19007048}, 
    number={4}, 
    journal={Compositio Mathematica}, 
    publisher={London Mathematical Society}, 
    author={Conti, Andrea}, year={2019}, 
    pages={776–831}
}

@article{DZ-bigimage,
    author = {Luis Dieulefait and Adrián Zenteno}, 
    title = {On the images of the Galois representations attached to generic automorphic representations of $\GSp(4)$}, 
    year = {2020}, 
    journal = {The Annali della Scuola Normale Superiore di Pisa, Classe di Scienze}, 
    volume = {XX}, 
    issue = {2}, 
    pages = {635 -- 655}, 
    doi = {https://doi.org/10.2422/2036-2145.201609_016}
}

@inbook{Hida-Tilouine_2015, 
    place={Cambridge}, 
    series={London Mathematical Society Lecture Note Series}, 
    title={Big image of Galois representations and congruence ideals}, 
    DOI={10.1017/CBO9781316106877.014}, 
    booktitle={Arithmetic and Geometry}, 
    publisher={Cambridge University Press}, 
    author={Hida, Haruzo and Tilouine, Jacques}, 
    editor={Dieulefait, Luis and Faltings, Gerd and Heath-Brown, D. R. and Manin, Yu. V. and Moroz, B. Z. and Wintenberger, Jean-PierreEditors}, 
    year={2015}, 
    pages={217 –- 254}, 
    collection={London Mathematical Society Lecture Note Series}
}

@article{Wu-pairing,
    author = {Ju-Feng Wu}, 
    title = {A pairing on the cuspidal eigenvariety for $\mathrm{GSp}_{2g}$ and the ramification locus},
    year = {2021}, 
    journal = {Documenta Mathematica}, 
    volume = {26},
    pages = {675 -- 711}, 
    doi = {https://doi.org/10.25537/dm.2021v26.675-711}
}

\vspace{15mm}

\begin{tabular}{l}
    University of Warwick   \\
    Mathematics Institute\\
    Coventry, UK\\
    \textit{E-mail address: }\texttt{ Ju-Feng.Wu@warwick.ac.uk }
\end{tabular}

\end{document}